\documentclass[10pt]{amsart}

\usepackage{amsmath}
\usepackage{amssymb}
\usepackage{amsthm}
\usepackage{braket}
\usepackage{enumerate} 
\usepackage{amsthm}
\usepackage{graphicx}
\usepackage{psfrag}
\usepackage{framed}
\usepackage[T1]{fontenc}
\input{xy}  
\xyoption{all} 
\CompileMatrices 
\usepackage{scalerel}

\theoremstyle{plain}
\newtheorem{theorem}{Theorem}
\newtheorem*{theorem*}{Theorem}
\newtheorem{lemma}[theorem]{Lemma}
\newtheorem{corollary}[theorem]{Corollary}
\newtheorem{proposition}[theorem]{Proposition}
\newtheorem{conjecture}[theorem]{Conjecture}

\theoremstyle{definition}

\newtheorem{definition}[theorem]{Definition}

\newtheorem{notation}[theorem]{Notation}

\numberwithin{equation}{section}
\usepackage{hyperref}

\newcommand{\eps}{\varepsilon}
\newcommand{\id}{\mathrm{id}}

\newcommand{\Z}{\mathcal Z}
\newcommand{\Cu}{\mathrm{Cu}}
\newcommand{\CU}{\mathbf{Cu}}
\newcommand{\supp}{\mathrm{Supp}}
\newcommand{\dimnuc}{\dim_{\mathrm{nuc}}}
\newcommand{\dr}{\mathrm{dr}}
\newcommand{\K}{\mathcal K}
\newcommand{\Q}{\mathcal Q}

\title{Nuclear dimension and $\Z$-stability}

\author{Yasuhiko Sato}
\address{\hskip-\parindent
Yasuhiko Sato, Graduate School of Science, Kyoto University, Sakyo-ku, Kyoto 606-8502, Japan.}
\email{ysato@math.kyoto-u.ac.jp}
\author{Stuart White}
\address{\hskip-\parindent
Stuart White, School of Mathematics and Statistics, University of Glasgow, 
University Gardens, Glasgow Q12 8QW, Scotland.}
\email{stuart.white@glasgow.ac.uk}
\author{Wilhelm Winter}
\address{\hskip-\parindent
Wilhelm Winter, Mathematisches Institut der WWU M\"unster, Einsteinstra\ss{}e 62, 48149 M\"unster, Germany.}
\email{wwinter@uni-muenster.de}
\thanks{Research partially supported by JSPS (the Grant-in-Aid for Research Activity Start-up 25887031), by EPSRC (grant no. I019227/1-2), and by the DFG (SFB 878).}

\date{\today}

\begin{document}
\begin{abstract}
Simple, separable, unital, monotracial and nuclear $\mathrm{C}^*$-algebras are shown to have finite nuclear dimension whenever they absorb the Jiang-Su algebra $\mathcal{Z}$ tensorially.  This completes the proof of the Toms-Winter conjecture in the unique trace case.
\end{abstract}

\maketitle

\renewcommand*{\thetheorem}{\Alph{theorem}}

\noindent
The structure theory of simple nuclear $\mathrm{C}^*$-algebras is currently undergoing revolutionary progress, driven by the discovery 
of regularity properties of various flavours: topological, functional analytic and algebraic. Despite the diverse nature of these regularity properties, they are all satisfied by those classes of $\mathrm{C}^*$-algebras which have been successfully classified by $K$-theoretic data, and they all fail spectacularly for the ``exotic'' algebras  in \cite{R:Acta,T:Ann} which provide counterexamples to Elliott's classification conjecture.  The observation that there are deep connections between these disparate properties was crystallised in the following conjecture of Toms and the third named author.

\begin{conjecture}\label{CA}
Let $A$ be a simple, separable, unital, nuclear and non-elementary $\mathrm{C}^*$-algebra.  Then the following statements are equivalent:
\begin{enumerate}[(i)]
\item $A$ has finite nuclear dimension;
\item $A$ absorbs the Jiang-Su algebra $\Z$ tensorially ($A$ is $\Z$-stable);
\item $A$ has strict comparison.
\end{enumerate}
\end{conjecture}

For now we only give a rough idea of what these notions mean; we will discuss them in greater detail in Section \ref{regularity-section}. 

Nuclear dimension is the $\mathrm{C}^{*}$-algebraic analogue of Lebesgue covering dimension as introduced in \cite{WZ:Adv}. This should be thought of as a topological property, phrased in terms of approximating $A$ via noncommutative partitions of unity; the numerical value of dimension enters as a colouring number for the latter. The notion is closely related to decomposition rank, its precursor from \cite{KW:IJM}. Having finite decomposition rank is a stronger condition than (i) which implies  quasidiagonality, and so (in contrast to nuclear dimension) is only applicable to stably finite algebras. 

A $\mathrm{C}^{*}$-algebra $A$ is $\mathcal{Z}$-stable if $A \cong A \otimes \mathcal{Z}$, where $\mathcal{Z}$ denotes the Jiang-Su algebra (which is nuclear, so there is no need to specify the tensor product). Jiang and Su constructed $\mathcal{Z}$ in \cite{JS:AJM} as an inductive limit of so-called dimension drop algebras (i.e.\ certain bundles of matrices over the closed unit interval). Amongst all such inductive limits, they characterised it as the unique one which is simple and monotracial. The classification machinery for nuclear $\mathrm{C}^{*}$-algebras allows for a much more general statement: $\mathcal{Z}$ is the uniquely determined infinite dimensional, simple and monotracial $\mathrm{C}^{*}$-algebra which has finite decomposition rank and is $KK$-equivalent to the complex numbers; cf.\ \cite[Corollary~5.5]{W:Invent1}. Here, $KK$-equivalence (in the sense of Kasparov) may be interpreted as a weak notion of homotopy equivalence. It implies in particular that $\mathcal{Z}$ and $\mathbb{C}$ have isomorphic ordered $K$-groups; in fact, in some sense $\mathcal{Z}$ plays the role of an infinite dimensional version of $\mathbb{C}$. Thanks to \cite{MS:Preprint}, one can also replace the hypothesis of finite decomposition rank in the statement above by finite nuclear dimension combined with quasidiagonality. The  Jiang--Su algebra is strongly self-absorbing in the sense of \cite{TW:TAMS}, i.e., there is an isomorphism $\mathcal{Z} \cong \mathcal{Z} \otimes \mathcal{Z}$ which is approximately unitarily equivalent to the first factor embedding. This is perhaps the most crucial feature of $\mathcal{Z}$, since it provides the link to celebrated results of Connes on injective II$_{1}$ factors and of Kirchberg on purely infinite nuclear $\mathrm{C}^{*}$-algebras (we will return to this point of view below). In \cite{W:JNCG}, $\mathcal{Z}$ was characterised in an entirely abstract manner as the initial object in the category of strongly self-absorbing $\mathrm{C}^{*}$-algebras. Being $\mathcal{Z}$-stable has a functional analytic flavour, which becomes particularly clear when characterising $\mathcal{Z}$-stability in terms of central sequence algebras; cf.\ Proposition~\ref{NewProp1.2}. 

Strict comparison means that positive elements (or rather their support projections) may be compared in the sense of Murray-von Neumann by looking at their values on traces. This can be rephrased in a much more algebraic manner in terms of a certain notion of order-completeness of a homological invariant (the so-called Cuntz semigroup). It is the different nature of these three regularity properties -- topological, functional analytic, and algebraic -- which makes the conjecture so useful, since it sheds light on the same phenomenon from completely different angles. Moreover, and quite strikingly, versions of these properties and their interplay also appear in other contexts, for example in von Neumann algebras and in topological dynamics. 

\bigskip

Up to now the Toms-Winter conjecture has been verified for various naturally occurring classes of algebras: in particular many $\mathrm{C}^*$-algebras of the form $C(X)\rtimes\mathbb Z$ arising from minimal homeomorphisms of compact metric spaces satisfy the conjecture. In the uniquely ergodic case Elliott and Niu have shown that regularity is automatic \cite{EN:meandim} (a fact related to unique ergodicity implying mean dimension zero -- regularity can fail for larger trace spaces). This leads to classification by $K$-theory (see \cite{TomsWinter:PNAS}, and \cite{Sz:Preprint, Win:class-products} for the case of $\mathbb{Z}^{d}$ actions). Moreover, the implications (i)$\Rightarrow$(ii) and (ii)$\Rightarrow$(iii) have been established in general by the third named author \cite{W:Invent1,W:Invent2} (see \cite{T:MA} for an extension to the stably projectionless case) and R\o{}rdam \cite{R:IJM} respectively; recent progress of Matui and the first named author \cite{MS:Acta} establishes (iii)$\Rightarrow$(ii) for $\mathrm{C}^*$-algebras with a unique tracial state (this has subsequently been extended to somewhat more general trace simplices in \cite{KR:Crelle,S:Preprint2,TWW:Preprint}). In the traceless case, (iii)$\Rightarrow$(ii) boils down to Kirchberg's celebrated $\mathcal{O}_{\infty}$-stability theorem for purely infinite nuclear $\mathrm{C}^{*}$-algebras, \cite{Kir:ICM}; (ii)$\Rightarrow$(i) follows from Kirchberg-Phillips classification combined with \cite{WZ:Adv} in the presence of the universal coefficient theorem (UCT), and was shown in general in \cite{MS:Preprint} (see also \cite{BEMSW:Z-2}).

The focus of this paper is the implication from (ii) to (i) in the tracial case. Classification results provide one route towards the implication (ii)$\Rightarrow$(i), via the strategy of comparing an algebra with a directly constructed model which has finite topological dimension. As well as passing through vast amounts of technical machinery, this approach is inevitably restricted to $\mathrm{C}^*$-algebras satisfying the still mysterious UCT.  A direct approach was pioneered by Tikuisis and the third named author in \cite{TW:APDE} (heavily based on the results of Kirchberg and R{\o}rdam in \cite{KirRor:pi3}), which proves that $\Z$-stable locally homogeneous algebras have finite decomposition rank (even outside the simple setting); however this too relies in an essential way on the existence of a concrete inductive limit structure.  A recent breakthrough was achieved by Matui and the first named author in \cite{MS:Preprint}, which, in the presence of a unique tracial state, establishes finite decomposition rank from quasdiagonality and $\Z$-stability.   Our main theorem, stated below, does not require a quasidiagonality assumption (which can often be hard to verify outside a concrete inductive limit setting) to establish finite nuclear dimension from $\mathcal{Z}$-stability.

\begin{theorem}\label{Main}
Let $A$ be a simple, separable, unital, nuclear and $\Z$-stable $\mathrm{C}^*$-algebra with a unique tracial state.  Then $\dimnuc(A)\leq 3$.
\end{theorem}

Upon combining Theorem \ref{Main} with the main results of \cite{MS:Acta}, and the general implications (i)$\Rightarrow$(ii) of \cite{W:Invent2} and (ii)$\Rightarrow$(iii) of \cite{R:IJM}, this confirms the Toms-Winter conjecture in the form stated above for $\mathrm{C}^*$-algebras with a unique  tracial state. Thus the difference between finite decomposition rank and finite nuclear dimension for unital, simple, separable and monotracial $\Z$-stable algebras is precisely quasidiagonality, and it remains an open question as to whether such algebras are automatically quasidiagonal. (Indeed this is a special case of the Blackadar-Kirchberg conjecture from \cite{BK:MA}, which asks whether every stably finite nuclear $\mathrm{C}^*$-algebra is quasidiagonal.)
\begin{corollary}
Conjecture \ref{CA} holds under the additional assumption that $A$ has at most one tracial trace.
\end{corollary}

\bigskip

In the rest of the introduction we discuss the analogies between the regularity properties in Conjecture \ref{CA} and properties of von Neumann factors of type II$_1$ and provide an architectural outline of the proof of Theorem \ref{Main}. 

 Each of the statements in the Toms-Winter conjecture is a natural analogue of a corresponding property for II$_1$ factors: strict comparison relates to the fact that the Murray-von Neumann lattice of projections in a II$_1$ factor is determined by the trace; tensorial absorption of $\Z$ to being a McDuff factor (absorbing the hyperfinite II$_1$ factor tensorially); and finite nuclear dimension to hyperfiniteness. For (separably acting) injective II$_1$ factors these properties are theorems (of course comparison holds for all II$_1$ factors), and all play roles in carrying through the implication from injectivity to hyperfiniteness in Connes' celebrated paper \cite{C:Ann}.  From this viewpoint, carrying through the implications (iii)$\Rightarrow$(ii)$\Rightarrow$(i) in the Toms-Winter conjecture provides a $\mathrm{C}^*$-algebraic version of Connes' characterisations of injectivity.

In his work Connes relies on a detailed analysis of automorphisms of II$_1$ factors, and the notion of an \emph{approximately inner flip} (the automorphism $x\otimes y\mapsto y\otimes x$ on the tensor square can be approximated by inner automorphisms) plays a key role. A crucial step in Connes' proof that injectivity implies hyperfiniteness is to show that any automorphism of an injective II$_1$ factor with separable predual is approximately inner \cite[Corollary 3.2, Theorem 5.3]{C:Ann} in the strong topology. In particular, it has strongly approximately inner flip \cite[Theorem 5.1 (7)$\Rightarrow$(3)]{C:Ann} and therefore is McDuff  \cite[Theorem 5.1 (3)$\Rightarrow$(2)]{C:Ann} in a particularly  strong fashion; from this Connes obtains hyperfiniteness of $M$ \cite[Theorem 5.1 (2)$\Rightarrow$(1)]{C:Ann}. Of course, by now two beautiful alternative proofs without using automorphisms have been obtained by Haagerup and by Popa, but for our purposes it will be more suitable to focus on the original strategy involving approximately inner flips.

Effros and Rosenberg initiated the study of approximately inner flips in the setting of $\mathrm{C}^*$-algebras in \cite{ER:PJM}. As they observe, an approximately inner flip is a strong requirement on a $\mathrm{C}^*$-algebra (implying simplicity, nuclearity and at most one trace), but several prominent examples enjoy such a flip, like uniformly hyperfinite (UHF) $\mathrm{C}^*$-algebras, the Jiang-Su algebra $\Z$ and the Cuntz algebras $\mathcal O_2$ and $\mathcal O_\infty$, all of which play key roles in the modern study of nuclear $\mathrm{C}^*$-algebras. Using the approximately inner flip, Effros and Rosenberg characterised the universal UHF algebra $\mathcal Q$ as the unique separable unital $\mathrm{C}^*$-algebra which is quasidiagonal, has an approximately inner flip and tensorially absorbs $\Q$ \cite[Theorem 5.1]{ER:PJM}, using a $\mathrm{C}^*$-analogue of Connes' implication \cite[Theorem 5.1 (2)$\Rightarrow$(1)]{C:Ann}.  With the benefit of hindsight, this result can be thought of as the precursor to the strategy of \cite{MS:Preprint} and our work here. Effros and Rosenberg compare the first factor embedding $A\mapsto (A\otimes \Q)_\omega$ of $A$ into the ultrapower of $A\otimes\Q$ with an embedding of $A$ into $1_A\otimes \Q_\omega\subset(A\otimes\Q)_\omega$ obtained from quasidiagonality. They then conjugate these two embeddings using the approximately inner flip on $A$, and obtain an approximately finite dimensional (AF) structure on $A$ directly from this, whence $A$ is isomorphic to $\Q$ by Elliott's classification theorem \cite{E:JA}. From today's perspective, Effros and Rosenberg's approach can be used to directly give approximations verifying decomposition rank zero (equivalent to being AF by \cite{KW:IJM}).

Here, it should be noted that there is a significant difference between approximately inner flips in von Neumann algebras and $\mathrm{C}^*$-algebras. In both contexts this implies amenability, but there are $K$-theoretic obstructions preventing a general simple unital nuclear $\mathrm{C}^*$-algebra with unique trace from having an approximately inner flip. Indeed AF algebras with an approximately inner flip must be UHF \cite{ER:PJM}. Thus a key idea is that of a ``2-coloured approximately inner flip'' for a $\mathrm{C}^*$-algebra (in the norm sense) as introduced in \cite{MS:Preprint}, and this notion applies to a broader class of algebras. Starting from Connes' approximately inner flip on an injective II$_1$ factor, such a 2-coloured approximately inner flip was established for simple separable unital and nuclear $\mathrm{C}^*$-algebras with a unique trace which absorb a UHF-algebra tensorially by using a 2-coloured approximate Murray-von Neumann equivalence, which in turn is inspired by Haagerup's approach \cite{H:JFA} to injectivity implies hyperfiniteness. This ingredient is structured so that  the flipping strategy gives the estimate $\dr(A)\leq 1$ for a quasidiagonal simple separable unital nuclear $\mathrm{C}^*$-algebra $A$ which is $\mathcal{Q}$-stable (i.e. $A \cong A \otimes \mathcal{Q}$) and has unique trace -- the two colours in the decomposition rank estimate arising from those in the approximately inner flip. The estimate $\dr(A)\leq 3$ (i.e.\ a $4$-coloured approximation) of \cite[Theorem 1.1]{MS:Preprint} in the $\mathcal{Z}$-stable case arises from using the UHF-stable approximations twice.

In both \cite{ER:PJM} and \cite{MS:Preprint} quasidiagonality provides $^*$-homomorphisms from a quasidiagonal $\mathrm{C}^*$-algebra $A$ into the ultrapowers $U_\omega$ of suitable UHF algebras (analogous to the asymptotic embeddings into the hyperfinite II$_1$ factor which arise from Condition (2) in \cite[Theorem 5.1]{C:Ann}). As Voiculescu shows in \cite{V:Duke}, the cone over any $\mathrm{C}^*$-algebra $A$ is quasidiagonal, so nontrivial order zero maps from $A$ into these ultrapowers always exist. We aim to use such  maps in place of the $^*$-homomorphisms used in \cite{ER:PJM,MS:Preprint}. As a byproduct we obtain an alternative proof of Voiculescu's result when $A$ is nuclear without elementary quotients, and has a separating family of tracial states. The argument uses recent work of Kirchberg and R\o{}rdam on the tracial ideal in an ultrapower \cite{KR:Crelle}, which also shows that the order zero maps $\phi:A\rightarrow U_\omega$ can be taken to preserve a fixed trace on $A$ (cf.\ Lin's work \cite{Lin:crossed-product-AF-embedding} on AF-embeddability of crossed products), and are unital modulo the tracial ideal.   We perform the flipping strategy working in an algebra of the form $(A\otimes A\otimes V)_\omega$, the presence of a UHF algebra $V$ providing the $2$-coloured approximately inner flip. We then use the extra properties of the map $\phi$ to obtain order zero maps $\Psi^{(0)},\Psi^{(1)}:(U\otimes V)_\omega\rightarrow\Z$ such that $\sum_{i=0}^1\Psi^{(i)}(\phi(1_A)\otimes 1_V)=1_{\Z_\omega}$ (for this we use a uniqueness result for certain positive contractions in $\Z_\omega$ proved using recent Cuntz-semigroup classification theorems \cite{CE:IMRN,R:Adv}); thus the resulting factorisation takes values in $A\otimes\Z$. This strategy of proving results for $\Z$-stable algebras via the UHF-stable case has its origins in \cite{Win:localizingEC}. The use of two maps $\Psi^{(i)}$ doubles the number of colours in the final nuclear dimension approximation, leading to a $4$-coloured factorisation: two  colours from the $\Psi^{(i)}$, multiplied by two colours from the $2$-coloured approximately inner flip.  Accordingly, when $A$ does have an approximately inner flip, for example when it is strongly self-absorbing, we get an estimate $\dimnuc(A)\leq 1$ in Theorem \ref{Main} (see Theorem \ref{ApproxFlipDim1}).

In Section \ref{regularity-section} we recall the notion of order zero maps and the regularity properties appearing in the Toms-Winter conjecture. In Section \ref{AOD} we examine asymptotically order zero maps and establish some technical preliminaries.  Section \ref{QD} provides our ``trace-preserving''-quasidiagonality of cones over tracial nuclear $\mathrm{C}^*$-algebras. In Section \ref{IFlip} we extract the $2$-coloured approximately inner flip for monotracial, simple, unital, nuclear and UHF-stable $\mathrm{C}^*$-algebras from the proof of \cite[Theorem 4.2]{MS:Preprint}, and record some of its consequences.  Section \ref{Shuff} provides the uniqueness theorem used in Section \ref{LastSect} to produce the two maps $\Psi^{(i)}$ and then prove Theorem \ref{Main}.

\renewcommand*{\thetheorem}{\roman{theorem}}
\numberwithin{theorem}{section}

\section{Order zero maps and regularity properties}\label{regularity-section}

\noindent
In this section we recall some basic facts about completely positive order zero maps, and the regularity properites which appear in Conjecture~\ref{CA}. 

\bigskip

Let $A$ and $B$ be $\mathrm{C}^*$-algebras. A completely positive (c.p.) map $\psi:A\rightarrow B$ is said to be \emph{order zero} if it preserves orthogonality, i.e.\ $\psi(e)\psi(f)=0$ whenever $ef=0$. The structure theorem from  \cite{WZ:MJM} (based on the respective result of Wolff from \cite{Wol:disjointness} for disjointness preserving bounded maps) for order zero maps $\psi:A\rightarrow B$ provides a \emph{supporting $^*$-homomorphism} $\pi$ from $A$ into the multiplier algebra $\mathcal M(\mathrm{C}^*(\psi(A)))\subseteq B^{**}$ and a positive element $h\in \mathcal M(\mathrm{C}^*(\psi(A)))\cap \psi(A)'$ such that $\psi(a)=\pi(a)h=h\pi(a)$ for all $a\in A$.  In the case that $A$ is unital (as it will be throughout the paper) one can take $h=\psi(1_A)$ so that $\psi(a)=\psi(1_A)\pi(a)$ for $a\in A$  and it follows that a c.p.\ map $\psi:A\rightarrow B$ is order zero if and only if
\begin{equation}\label{OrderZeroRelation}
\psi(x)\psi(y)=\psi(xy)\psi(1_A),\quad x,y\in A.
\end{equation}
From the structure theorem one obtains two further key properties:
\begin{itemize}
\item the functional calculus for completely positive and contractive (c.p.c.) order zero maps \cite[Corollary 4.2]{WZ:MJM}: given $f\in C_0(0,1]_+$, define an order zero map $f(\psi):A\rightarrow B$ by $f(\psi)(a)=\pi(a)f(h)$ which is $\pi(a)f(\psi(1_A))$ when $A$ is unital. 
\item the duality between c.p.c.~ order zero maps $\psi:A\rightarrow B$ and $^*$-homomorphisms $C_0(0,1]\otimes A\rightarrow B$ on the cone over $A$, \cite[Corollary 4.1]{WZ:MJM}.
\end{itemize}

\bigskip

The nuclear dimension of a $\mathrm{C}^*$-algebra from \cite{WZ:Adv} is defined in terms of c.p.\ approximations which are uniformly decomposable into sums of order zero maps (cf.\ \cite{HKW:Adv}, which shows that approximations as below can be found for any nuclear $\mathrm{C}^*$-algebra when  $n$ is allowed to vary with $i$). The notion is inspired by the idea of regarding a nuclear $\mathrm{C}^{*}$-algebra as a noncommutative topological space; completely positive approximations then play the role of noncommutative partitions of unity.
\begin{definition}[{\cite[Definition 2.1]{WZ:Adv}, \cite[Definition 3.1]{KW:IJM}}]\label{DefND}
A $\mathrm{C}^*$-algebra $A$ has \emph{nuclear dimension} at most $n\in\mathbb N$, written $\dimnuc(A)\leq n$, if there exists a net $(F_i,\psi_i,\phi_i)_{i\in I}$ where $F_i$ are finite dimensional $\mathrm{C}^*$-algebras, which decompose as $F_i=F_i^{(0)}\oplus \dots \oplus F_i^{(n)}$, $\psi_i:A\rightarrow F_i$ are c.p.c.~ and $\phi_i:F_i\rightarrow A$ are c.p.\ such that $\lim_i\phi_i(\psi_i(a))=a$ for $a\in A$, and each $\phi_i|_{F^{(j)}_i}$ is contractive and order zero. If additionally each $\phi_i$ is contractive, then $A$ is said to have \emph{decomposition rank} at most $n$.
\end{definition}

\bigskip

The Jiang-Su algebra $\mathcal{Z}$ was constructed in \cite{JS:AJM} as an inductive limit of prime dimension drop intervals. It is simple, projectionless, has a unique tracial state and is strongly self-absorbing in the language of \cite{TW:TAMS}, i.e.~ $\Z\cong \Z\otimes\Z$ via an isomorphism $\theta:\Z\rightarrow\Z\otimes\Z$ which is approximately unitarily equivalent to the first factor embedding $a\mapsto a\otimes 1_\Z$, \cite[Theorem 7.6, Theorem 8.7]{JS:AJM}. There are by now several alternative ways (ranging from very concrete to very abstract) of characterising $\mathcal{Z}$. Likewise, $\mathcal{Z}$-stability, i.e.\ the property of absorbing the Jiang-Su algebra tensorially, can be expressed in quite different fashions. We recall a result which is particularly useful in this context, and which does not require the definition of the algebra $\mathcal{Z}$ itself. The idea is to realise a finite set of generators and relations inside the central sequence algebra (of course, the latter may also be replaced by an ultrapower; cf.\ \ref{ultrapower-notation}). The statement essentially combines \cite[Proposition 5.1]{RW:Crelle} with \cite[Proposition 2.2]{TW:CJM}; an approximate version was given in \cite[Proposition 2.3]{W:Invent1}.
\begin{proposition}\label{NewProp1.2}
A separable unital $\mathrm{C}^{*}$-algebra $A$ is $\mathcal{Z}$-stable if and only if, for any $2 \le p \in \mathbb{N}$, there are c.p.c.\ order zero maps
\begin{equation}
\Phi:M_{p} \longrightarrow \big(\ell^{\infty}(A)/c_{0}(A)\big) \cap A'
\end{equation}
and
\begin{equation}
\Psi:M_{2} \longrightarrow \big(\ell^{\infty}(A)/c_{0}(A)\big) \cap A'
\end{equation}
satisfying the relations 
\begin{equation}
 \Psi\big(e^{(2)}_{11}\big) + \Phi\big(1_{M_{p}}\big) = 1_{\ell^{\infty}(A)/c_{0}(A)} \mbox{ and }\Psi\big(e^{(2)}_{22}\big) \Phi\big(e^{(p)}_{11}\big) = \Psi\big(e^{(2)}_{22}\big).
\end{equation}
\end{proposition}
Note that this result is in complete analogy to McDuff's characterisation of $\mathcal{R}$-stability of II$_1$ factors in terms of the existence of approximately central copies of matrix algebras \cite{M:PLMS}, replacing a unital $^{*}$-homomorphism from matrices by an order zero map $\Phi$ which is large, in that $1-\Phi(1_{M_p})$ is dominated by $\Phi(e_{11}^{(p)})$ (as witnessed by $\Psi$).

\bigskip

For $D$ a $\mathrm{C}^*$-algebra with positive elements $x,y\in D\otimes\K$ write $x\precsim y$ when there exists a sequence $(v_n)_{n=1}^\infty$ in $D\otimes\K$ with $v_n^*yv_n\rightarrow x$.  Write $x\sim y$ when $x\precsim y$ and $y\precsim x$. The \emph{Cuntz semigroup} is defined to be $\Cu(D)=(D\otimes\K)_+/_{\sim}$, and we write $\langle x\rangle$ for the class of $x\in (D\otimes \K)_+$. This is equipped with an addition arising from identification of $\mathcal K\cong \mathcal K\otimes M_2$ to take an orthogonal sum of two positive elements of $D\otimes\mathcal K$, an order $\leq$ induced from $\precsim$, and has the property that every increasing sequence in $\Cu(D)$ has a supremum (\cite{CEI:Crelle}). An abstract category $\CU$ containing the Cuntz semigroups of $\mathrm{C}^*$-algebras was set out in \cite{CEI:Crelle}, which shows that the assignment $\Cu(\cdot)$ is functorial and preserves sequential inductive limits. Given a positive element $x\in A$ and $\eps>0$, write $(x-\eps)_+$ for $h_\eps(x)$ where $h_\eps(t)=\max(t-\eps,0)$.  Positive elements $x,y\in D\otimes\K$ satisfy $x\precsim y$ if and only if $(x-\eps)_+\precsim y$ for all $\eps>0$. This last result can be found in the survey article \cite{APT:Contemp} (as Proposition 2.17), to which we refer for a full account of the Cuntz semigroup and the category $\CU$.

\begin{definition}
A simple, separable, unital $\mathrm{C}^{*}$-algebra $D$ is said to have strict comparison if the following holds:

Whenever there are positive nonzero elements $a,b \in D \otimes \mathcal{K}$ such that 
\begin{equation}
\lim_{k \to \infty} \tau(a^{\frac{1}{k}}) < \lim_{k \to \infty} \tau(b^{\frac{1}{k}})
\end{equation} for any $2$-quasitrace $\tau$ on $D$, then $a\precsim b$.
\end{definition}

A (2-quasi)trace $\tau$ on a unital $\mathrm{C}^*$-algebra $D$ defines a functional $d_\tau$ on $\Cu(D)$ by $d_\tau(\langle a\rangle)=\mu_a((0,\infty))$ for $a\in (D\otimes\K)_+$, where $\mu_{a}$ is the measure on $[0,\infty)$ induced by $\mu_{a}(f)=\tau(f(a))$ for $f\in C_0[0,\infty)_+$. In particular, note that $d_\tau(\langle f(a)\rangle)=\mu_a(\supp(f))$ where $\supp(f)=\{t:f(t)\neq 0\}$ is the open support of $f\in C(\mathrm{Sp}(a))$. (We only need the definition of $d_\tau$ when $a\in D$, but in general $\tau$ is extended to a lower semicontinuous ($2$-quasi)trace on $D\otimes\mathcal K$ to make this definition.)  With this terminology, a simple, separable, unital $\mathrm{C}^*$-algebra $D$ has strict comparison if and only if $d_\tau(\langle a\rangle)<d_\tau(\langle b\rangle)$ for all $2$-quasitraces $\tau$ implies that $a\precsim b$. It should be pointed out that we will only encounter strict comparison in situations when it is known that all $2$-quasitraces are in fact traces.

In order to appreciate the algebraic flavour of strict comparison, consider R{\o}rdam's characterisation in terms of an order property of the Cuntz semigroup $\mathrm{Cu}(A)$; see \cite{R:IJM} or \cite[Section 5]{APT:Contemp}.    
\begin{proposition}
A simple, separable, unital, nuclear $\mathrm{C}^{*}$-algebra $A$ has strict comparison if and only if its Cuntz semigroup $\mathrm{Cu}(A)$ is almost unperforated, i.e., $(n+1) \cdot \langle a \rangle \le n \cdot \langle b \rangle$ for some $n\in\mathbb N$ implies $\langle a \rangle \le \langle b \rangle$. 
\end{proposition}

\section{Asymptotically order zero maps}\label{AOD}

\noindent
We will be repeatedly concerned with sequences of asymptotically order zero maps inducing order zero maps into ultrapowers.  Firstly we set out our notation for working with the latter.
\begin{notation}\label{ultrapower-notation}
Throughout the paper $\omega$ will denote a free ultrafilter on $\mathbb N$. Given a $\mathrm{C}^*$-algebra $A$, write $A_\omega$ for the ultrapower obtained as the quotient $\ell^\infty(A)/c_\omega(A)$, where $c_\omega(A)=\{(x_n)_{n=1}^\infty\in\ell^\infty(A):\lim_{n\rightarrow\omega}\|x_n\|=0\}$.  We shall generally suppress the quotient map $\ell^\infty(A)\rightarrow A_\omega$ and instead say that $(x_n)_{n=1}^\infty$ \emph{represents} or \emph{lifts} $x\in A_\omega$ if its image in the quotient is $x$. It is well known that positive contractions can be lifted to sequences of positive contractions, and unitaries to sequences of unitaries.   We regularly regard $A$ as embedded in $A_\omega$, where $a\in A$ is represented by the sequence with constant value $a$.  We also work regularly with ultrapowers of tensor products of the form $(A\otimes B)_{\omega}$ and throughout the paper $\otimes$ denotes the spatial tensor product.\end{notation}

Given $\mathrm{C}^*$-algebras $A$ and $B$, a uniformly bounded sequence $(\psi_n)_{n=1}^\infty$ of bounded linear maps $:A\rightarrow B$ induces a bounded linear map $\Psi:A_\omega\rightarrow B_\omega$ defined at the level of representatives by $(x_n)_{n=1}^\infty\mapsto (\psi_n(x_n))_{n=1}^\infty$.  When each $\psi_n$ is completely positive, contractive, a $^*$-homomorphism, or order zero, then $\Psi$ enjoys the same property.  When each $\psi_n$ is c.p.\ then the induced map $\Psi$ is order zero if and only if $
\lim_{n \rightarrow \omega} \|\psi_{n}(x_{n}) \psi_{n}(y_{n}) \| = 0$ whenever $\lim_{n \rightarrow \omega} \|x_{n} y_{n} \| = 0$ for bounded sequences $(x_{n})_{n=1}^{\infty}, (y_{n})_{n=1}^{\infty}$ in $A$. 
Sometimes the map $\Psi$ will only be order zero when restricted to a subalgebra, say $A \subset A_{\omega}$. This is the case if and only if the sequence is asymptotically order zero with respect to $A$ in the sense that $\lim_{n \rightarrow \omega} \|\psi_{n}(x) \psi_{n}(y) \| = 0$ whenever $x,y \in A$ satisfy $xy = 0$; when $A$ is unital we can describe this by asking for (\ref{OrderZeroRelation}) to hold asymptotically, i.e.
\begin{equation}
\lim_{n\rightarrow\omega}\|\psi_n(xy)\psi_n(1)-\psi_n(x)\psi_n(y)\|=0,\quad x,y\in A.
\end{equation}
Note too that when $A$ is nuclear, by the Choi-Effros lifting theorem \cite{CE:Ann} any c.p.c.~ order zero map $A\rightarrow B_\omega$ arises from a sequence of c.p.c.~ maps $A\rightarrow B$ which is asymptotically order zero with respect to $A$.  Sequences of asymptotically order zero maps arise naturally from the approximations in the definition of finite nuclear dimension; indeed, given a system of maps $(F_i,\psi_i,\phi_i)_{i\in I}$ as in Definition \ref{DefND} one can modify the approximation by removing unnecessary components so that the maps $\psi_i$ are asymptotically order zero \cite[Proposition 3.2]{WZ:Adv}.

Sequences of asymptotically order zero maps enjoy a slight variant of the structure theorem using a supporting order zero map in place of the supporting $^{*}$-homomorphism. A crucial point here is that the supporting order zero map can be chosen to take values in the ultrapower $B_\omega$, rather than the multiplier algebra $\mathcal M(\mathrm{C}^*(\psi(A))) \subset (B_{\omega})^{**}$.  

\begin{lemma}\label{SequenceStructureTheorem}
Let $A$ and $B$ be separable unital $\mathrm{C}^*$-algebras, and $(\psi_n)_{n=1}^\infty$ a sequence of c.p.c.~ maps $A\rightarrow B$ inducing a c.p.c.~ map $\Psi:A_\omega\rightarrow B_\omega$.  Let $X$ be a separable $\mathrm{C}^*$-subalgebra of $A_\omega$  with $1_{A_\omega}=1_X$ and suppose that $\Psi|_X:X\rightarrow B_\omega$ is order zero.   Then there exists a sequence $(\pi_n)_{n=1}^\infty$ of c.p.c.~ maps $\pi_n:A\rightarrow \mathrm{C}^*(\psi_n(A))\subseteq B$ such that the induced c.p.c.~ map $\Pi:A_\omega\rightarrow B_\omega$ restricts to an order zero map on $X$ satisfying
\begin{equation}\label{SequenceStructureTheorem.Main}
\Psi(x)=\Pi(x)\Psi(1_X)=\Psi(1_X)\Pi(x),\quad x\in X.
\end{equation}

Suppose additionally that $A$ and $B$ factorise as $A=\widehat{A}\otimes C$ and $B=\widehat{B}\otimes C$, for some separable unital $\mathrm{C}^*$-algebra $C$, and each $\psi_n$ is equal to $\widehat{\psi}_n\otimes\id_C$ for some c.p.c.~ map $\widehat{\psi}_n:\widehat{A}\rightarrow \widehat{B}$.  Then each $\pi_n$ may be chosen of the form $\widehat{\pi}_n\otimes\id_C$ for some c.p.c.~ map $\widehat{\pi}_n:\widehat{A}\rightarrow C^*(\widehat{\psi}_n(\widehat{A}))$.
\end{lemma}

\begin{proof}
For $m\in\mathbb N$, define piecewise linear functions $g_m\in C_0(0,1]$ by
\begin{equation}
g_m(t)=\begin{cases}1,&t\geq 1/m;\\2mt-1,&1/(2m)\leq t\leq 1/m;\\0,&0\leq t\leq 1/(2m).\end{cases}
\end{equation}
Define $h_m\in C_0(0,1]_+$ to satisfy $h_m(t)t=g_m(t)$ for all $t\in (0,1]$.  The c.p.c.~ order zero map $g_m(\Psi|_X):X\rightarrow  B_\omega \cap \Psi(1_{X})'$ takes the form
\begin{equation}
g_m(\Psi|_X)(x)=h_m(\Psi(1_X))^{1/2}\Psi(x)h_m(\Psi(1_X))^{1/2},\quad x\in X,
\end{equation}
since if $\rho$ is the supporting homomorphism of $\Psi|_X$, then
\begin{align}
g_m(\Psi|_X)(\cdot)&=\rho(\cdot)g_m(\Psi|_X(1_X))\nonumber\\&=\rho(\cdot)\Psi|_X(1_X)h_m(\Psi|_X(1_X))=\Psi|_X(\cdot)h_m(\Psi|_X(1_X)).
\end{align}

Fix a countable dense sequence $(x^{(i)})_{i=1}^\infty$ in the unit ball of $X$ and lift each $x^{(i)}$ to a sequence $(x^{(i)}_n)_{n=1}^\infty$ of contractions in $A$.   Since $(h_m(\psi_n(1_A))^{1/2})_{n=1}^\infty$ is a lift of $h_m(\Psi(1_X))^{1/2}$ the map $g_m(\Psi|_X)$ is (the restriction to $X$ of the c.p.c.~ map) induced by the sequence of c.p.c.~ maps $(\widetilde{\psi}_{m,n})_{n=1}^\infty$ from $A$ to $B$ given by
\begin{equation}
\widetilde{\psi}_{m,n}(a)=h_m(\psi_n(1_A))^{1/2}\psi_n(a)h_m(\psi_n(1_A))^{1/2},\quad a\in A.
\end{equation}
It then follows that for each $i,j,m\in\mathbb N$,
\begin{equation}\label{SequenceStructureTheorem.2}
\lim_{n\rightarrow\omega}\|\widetilde{\psi}_{m,n}(x^{(i)}_nx^{(j)}_n)\widetilde{\psi}_{m,n}(1_A)-\widetilde{\psi}_{m,n}(x^{(i)}_n)\widetilde{\psi}_{m,n}(x^{(j)}_n)\|=0.
\end{equation}
As $g_m(\Psi|_X)(X)\subset B_\omega \cap \Psi(1_X)'$, for all $i,m\in\mathbb N$ we also have
\begin{equation}\label{SequenceStructureTheorem.3}
\lim_{n\rightarrow\omega}\|\widetilde{\psi}_{m,n}(x^{(i)}_n)\psi_n(1_A)-\psi_n(1_A)\widetilde{\psi}_{m,n}(x^{(i)}_n)\|=0.
\end{equation}

For each $n\in\mathbb N$, let $m_n$ denote the maximum value of $m\in\{1,2,\dots,n\}$ such that
\begin{equation}\label{SequenceStructureTheorem.1}
\|\widetilde{\psi}_{m,n}(x^{(i)}_nx^{(j)}_n)\widetilde{\psi}_{m,n}(1_A)-\widetilde{\psi}_{m,n}(x^{(i)}_n)\widetilde{\psi}_{m,n}(x^{(j)}_n)\|\leq\frac{2}{m},\quad 1\leq i,j\leq m,
\end{equation}
and
\begin{equation}\label{SequenceStructureTheorem.4}
\|\widetilde{\psi}_{m,n}(x^{(i)}_n)\psi_n(1_A)-\psi_n(1_A)\widetilde{\psi}_{m,n}(x^{(i)}_n)\|\leq \frac{2}{m},\quad 1\leq i\leq m.
\end{equation}
The limits in (\ref{SequenceStructureTheorem.2}) and (\ref{SequenceStructureTheorem.3}) demonstrate that $\lim_{n\rightarrow\omega}m_n=\infty$.

Define $\pi_n=\widetilde{\psi}_{m_n,n}:A\rightarrow B$.  As $\lim_{n\rightarrow\omega}m_n=\infty$, (\ref{SequenceStructureTheorem.1}) and the order zero relation (\ref{OrderZeroRelation}) show that the sequence $(\pi_n)_{n=1}^\infty$ induces a c.p.c.~ order zero map $\Pi:X\rightarrow B_\omega$ and (\ref{SequenceStructureTheorem.4}) shows that $\Pi(X)\subset\Psi(1_X)'$.  Denote by $e\in B_\omega$ the contraction represented by the sequence $(g_{m_n}(\psi_n(1_A))_{n=1}^\infty$.  As $\|g_m(\psi_n(1_A))\psi_n(1_A)-\psi_n(1_A)\|\leq 2/m$ for all $n$ and $\lim_{n\rightarrow\omega}m_n=\infty$. It follows that $e\Psi(1_X)=\Psi(1_X)=\Psi(1_X)e$. Since $\Psi(X)$ lies in the hereditary subalgebra $\overline{\Psi(1_X)B_\omega\Psi(1_X)}$, we have 
\begin{equation}\label{SequenceStructureTheorem.5}
e\Psi(x)=\Psi(x)e=\Psi(x),\quad x\in X.
\end{equation}
Now, given $x\in X$ with lift $(x_n)_{n=1}^\infty$, we have
\begin{equation}
\psi_n(1_A)^{1/2}\pi_n(x_n)\psi_n(1_A)^{1/2}=g_{m_n}(\psi_{n}(1_A))^{1/2}\psi_n(x_n)g_{m_n}(\psi_{n}(1_A))^{1/2},
\end{equation}
so (\ref{SequenceStructureTheorem.5}) gives
\begin{equation}
\Psi(1_X)^{1/2}\Pi(x)\Psi(1_X)^{1/2}=\Psi(x).
\end{equation}
As $\Pi(X)$ commutes with $\Psi(1_X)$, this establishes (\ref{SequenceStructureTheorem.Main}).

Now suppose that $A$, $B$ and each $\psi_n$ factorise as described in the second paragraph of the lemma.  Then $h_m(\psi_n(1_A))^{1/2}=h_m(\widehat{\psi}_n(1_{\widehat{A}}))^{1/2}\otimes 1_C$, and so each $\widetilde{\psi}_{m,n}$ factorises as $\check{\psi}_{m,n}\otimes\id_C$, with $\check{\psi}_{m,n}:\widehat{A}\rightarrow \mathrm{C}^*(\widehat{\psi}_n(\widehat{A}))$ given by
\begin{equation}
\check{\psi}_{m,n}(a)=h_m(\widehat{\psi}_n(1_{\widehat{A}}))^{1/2}\widehat{\psi}_n(a)h_m(\widehat{\psi}_n(1_{\widehat{A}}))^{1/2},\quad a\in \widehat{A}.
\end{equation}
Thus each $\pi_n$ also enjoys the specified factorisation.
\end{proof}

The functional calculus for order zero maps $\Psi:X\rightarrow B_\omega$ arising as in Lemma \ref{SequenceStructureTheorem} can be recaptured using the supporting order zero map $\Pi$: for $f\in C_0(0,1]_+$ and $x\in X$, we have $f(\Psi)(x)=\Pi(x)f(\Psi(1_X))$. When $f$ is a polynomial with $f(0)=0$,  then $f(t) = t \cdot g(t)$ for some polynomial $g$ and we have $f(\Psi)(x) = \bar{\Psi}(x) f(\Psi(1_{X})) = \bar{\Psi}(x) \Psi(1_{X}) g(\Psi(1_{X})) = \Psi(x) g(\Psi(1_{X})) = \Pi(x) \Psi(1_{X}) g(\Psi(1_{X})) = \Pi(x) f(\Psi(1_{X}))$, where $\bar{\Psi}$ denotes the supporting $^{*}$-homomorphism of $\Psi$. In general the assertion follows by approximating $f$ uniformly by polynomials. In particular, this shows that $f(\Psi)$ is induced by a sequence of asymptotically order zero maps, whose ranges lie in the same $\mathrm{C}^*$-algebras as the maps used to obtain $\Psi$.
\begin{lemma}\label{SequenceFunctionalCalculus}
Let $A$ and $B$ be separable unital $\mathrm{C}^*$-algebras and $(\psi_n)_{n=1}^\infty$ a sequence of c.p.c.~ maps $A\rightarrow B$ inducing a c.p.c.~ map $\Psi:A_\omega\rightarrow B_\omega$.  Let $X$ be a separable $\mathrm{C}^*$-subalgebra of $A_\omega$ with $1_X=1_{A_\omega}$, and suppose that $\Psi|_X:X\rightarrow B_\omega$ is order zero.  Given any function $f\in C_0(0,1]_+$, the order zero map $f(\Psi|_X):X\rightarrow B_\omega$ is (the restriction to $X$ of the c.p.c.~ map) induced by a sequence $(\widetilde{\psi}_n)_{n=1}^\infty$ of c.p.\ maps $\widetilde{\psi}_n:A\rightarrow C^*(\psi_n(A))\subseteq B$.

Suppose additionally that $A$ and $B$ factorise as $A=\widehat{A}\otimes C$ and $B=\widehat{B}\otimes C$, for some separable unital $\mathrm{C}^*$-algebra $C$, and each $\psi_n$ is equal to $\widehat{\psi}_n\otimes\id_C$ for some c.p.c.~ map $\widehat{\psi}_n:\widehat{A}\rightarrow \widehat{B}$.  Then each $\widetilde{\psi}_n:A\rightarrow B$ may be chosen of the form $\check{\psi}_n\otimes\id_C$ for some c.p.c.~ map $\check{\psi}_n:\widehat{A}\rightarrow C^*(\widehat{\psi}_n(\widehat{A}))$.
\end{lemma}
\begin{proof}
This follows from the expression $f(\Psi|_X)(x)=\Pi(x)f(\Psi(1_X))$ for $x\in X$, where $\Pi$ is the supporting order zero map of Lemma \ref{SequenceStructureTheorem} induced by c.p.c.~ maps $\pi_n:A\rightarrow \mathrm{C}^*(\psi_n(A))$.
\end{proof}

We end the section by recording a fact about polar decompositions in ultrapowers. The proof that unitary polar decompositions exist given in \cite{L:Book} is stated for quotients $\prod B_n/\sum B_n$, but applies without modification to ultrapowers. It then follows immediately that $B_\omega$ has stable rank one ($x=u|x|$ can be approximated by invertibles of the form $u(|x|+\eps1_{B_\omega})$).
\begin{lemma}[cf.\ {\cite[Lemma 19.2.2(1)]{L:Book}}]\label{UnitaryPolar}
Let $B$ be a separable unital $\mathrm{C}^*$-algebra with stable rank one. Then for every $x\in B_\omega$, there exists a unitary $u\in B_\omega$ with $x=u|x|$ and so $B$ has stable rank one.
\end{lemma}

\section{Tracially large UHF embeddings of cones}\label{QD}

\noindent
In this section we provide the tracially large order zero maps $A\rightarrow U_\omega$ used to obtain the finite dimensional algebras in the finite nuclear dimension factorisations of our main result. The proof works when $U$ is replaced by any separable, unital, nuclear and stably finite $\mathrm{C}^*$-algebra.

\begin{notation}
Let $A$ be a separable unital $\mathrm{C}^*$-algebra with a fixed tracial state $\tau$, and let $\tau_{A,\omega}$ be the induced tracial state on $A_\omega$ defined on a representative sequence $(x_n)_{n=1}^\infty$ by $\lim_{n\rightarrow\omega}\tau(x_n)$.  Following \cite{KR:Crelle}, we define the trace kernel ideal of $\tau_{A,\omega}$ as  $J_{A,\tau} := \{ (a_{n})_{n} \in A_{\omega} \mid \lim_{n \to \omega}  \tau (a_{n}^{*}a) \} \lhd A_\omega$. (We do not keep track of the ultrafilter in our notation for the ideal. However, this should not cause confusion; in fact, we think of the free ultrafilter as being fixed throughout the paper.)
\end{notation}

We use Kirchberg's notion of a $\sigma$-ideal from \cite{K:Abel}, working with the formulation of  \cite[Definition 4.4]{KR:Crelle}: an ideal $I$ in a $\mathrm{C}^*$-algebra $D$ is a \emph{$\sigma$-ideal} if for every separable $\mathrm{C}^*$-subalgebra $C$ of $D$ there exists a positive contraction $e\in I \cap C'$ such that $ec=c$ for all $c\in I \cap C$.  A key observation of Kirchberg and R\o{}rdam is that $J_{A,\tau}$ is a $\sigma$-ideal in $A_\omega$ whenever $A$ is unital, and $\tau$ is a tracial state on $A$, \cite[Proposition 4.6 and Remark 4.7]{KR:Crelle}. (This is a special case of the results of \cite[Section 4]{KR:Crelle}; the trace-kernel ideal $J_A$ obtained from those sequences which converge to zero uniformly on all traces is also a $\sigma$-ideal in $A_\omega$, and so too are the ideals $(J_A\cap A')\lhd (A_\omega \cap A')$ and $(J_{A,\tau}\cap A')\lhd(A_\omega\cap A')$ in the central sequence algebras).  The proof below is based on the argument for surjectivity given at the end of Section 4 of \cite{KR:Crelle}. In forthcoming work we will provide some uniqueness statements for such maps up to a finitely coloured decomposition as a sum of order zero maps arising from conjugation by a suitable sequence of contractions.
 \begin{proposition}\label{EmbeddingCone}
 Let $A$ and $B$ be separable unital and nuclear $\mathrm{C}^*$-algebras without elementary quotients, and with extremal traces $\tau_A$ and $\tau_B$ respectively.  Then there exists a c.p.c.~ order zero map $\Psi:A\rightarrow B_\omega$ with $1_{B_\omega}-\Psi(1_A)\in J_{B,\tau_B}$ and $\tau_{B,\omega}(\Psi(x))=\tau_A(x)$ for all $x\in A$.
\end{proposition}
\begin{proof}
The strong operator closures of $A$ and $B$ in the GNS representations corresponding to $\tau_A$ and $\tau_B$ are injective II$_1$ factors, so by Connes' uniqueness theorem \cite{C:Ann}, both these strong closures are isomorphic to the hyperfinite II$_1$ factor $R$.  Let $R^\omega$ be the von Neumann algebra ultrapower, namely the quotient of $\ell^\infty(R)$ by those sequences $(x_n)_{n=1}^\infty$ with $\lim_{n\rightarrow\omega}\tau_R(x_n^*x_n)=0$.  By Kaplansky's density theorem, the canonical emebddings $A_\omega/J_{A,\tau_A}\rightarrow R^\omega$ and $B_\omega/J_{B,\tau_B}\rightarrow R^\omega$ are surjective, so give a surjective $^*$-isomorphism $\widetilde{\Psi}:A_\omega/J_{A,\tau_A}\rightarrow B_\omega/J_{B,\tau_B}$. Let $\pi:A\rightarrow A_\omega/J_{A,\tau_A}$ be the canonical unital $^*$-homomorphism. Applying the Choi-Effros lifting theorem \cite{CE:Ann}, we obtain a unital c.p.c.~ map $\widehat{\Psi}:A\rightarrow B_\omega$ lifting $\widetilde{\Psi}\circ\pi$, i.e.~ $\widetilde{\Psi}\circ\pi=q\circ\widehat{\Psi}$, where $q:B_\omega\rightarrow B_\omega/J_{B,\tau_B}$ is the quotient map.  As $J_{B,\tau_B}$ is a $\sigma$-ideal in $B_\omega$, there is a positive contraction $e\in J_{B,\tau_B}\cap \mathrm{C}^*(1_{B_{\omega}}, \widehat{\Psi}(A))'$ with $ec=c$ for all $c\in  J_{B,\tau_B} \cap \mathrm{C}^*(1_{B_{\omega}}, \widehat{\Psi}(A))$.  

Define a c.p.c.~ map $\Psi:A\rightarrow B_\omega$ by $\Psi(a)=(1_{B_\omega}-e)\widehat{\Psi}(x)(1_{B_\omega}-e)$ for $x\in A$.  We have $1_{B_\omega}-\Psi(1_A)=  1_{B_\omega}-(1_{B_\omega}-e)^{2}=2e - e^{2}\in J_{B,\tau_B}$ and
\begin{equation}
\tau_{B,\omega}(\Psi(x))=\tau_{B,\omega}(\widehat{\Psi}(x))=\tau_{B,\omega}(\widetilde{\Psi}\circ\pi(x))=\tau_{A,\omega}\circ\pi(x)=\tau_A(x),\quad x\in A;
\end{equation}
for the third equation note that since $R^{\omega}$ is a II$_1$ factor (see \cite[Lemma A.4.2]{SS:Book}) it has a unique tracial state (any trace on a II$_1$ factor is automatically normal \cite[Proposition V.2.5]{T:1}), and so we have $\tau_{B,\omega}(\widetilde{\Psi}(x)) = \tau_{A,\omega}(x)$ for any $x \in A_{\omega}/J_{A,\tau_{A}}$.

Finally, if $x,y\in A_+$ satisfy $xy=0$, then $\widehat{\Psi}(x)\widehat{\Psi}(y)\in \ker q=J_{B,\tau_B}$, and hence $(1_{B_\omega}-e)\widehat{\Psi}(x)\widehat{\Psi}(y)=0$. Since $(1_{B_\omega}-e)$ commutes with $\widehat{\Psi}(A)$, we have
\begin{align}
\Psi(x)\Psi(y)&=(1_{B_\omega}-e)\widehat{\Psi}(x)(1_{B_\omega}-e)^2\widehat{\Psi}(y)(1_{B_\omega}-e)\nonumber\\
&=(1_{B_\omega}-e)^{2}\widehat{\Psi}(x)\widehat{\Psi}(y)(1_{B_\omega}-e)^2=0,
\end{align}
showing that $\Psi$ is order zero, as required.
\end{proof}
One can prove more. Since $J_{B,\tau_B} \cap B'$ is a $\sigma$-ideal in $B_\omega \cap B'$, starting with a trace preserving embedding $A\rightarrow R^\omega\cap R'$ (which is easily obtained from the isomorphism of $R$ with its infinite tensor product), and using the fact that $(B_\omega \cap B')/(J_{B,\tau_B} \cap B')\cong R^\omega\cap R'$ (\cite[Lemma 2.1]{S:Preprint} when $B$ is nuclear and \cite[Theorem 3.3]{KR:Crelle} for general $B$), the same proof allows us to insist that the order zero map $\Psi$ maps into the central sequence algebra $B_\omega \cap B'$.

Note too that the previous proposition recaptures Voiculescu's result \cite{V:Duke} that the cone over a $\mathrm{C}^*$-algebra $A$ is quasidiagonal in the special case when $A$ is unital and nuclear without elementary quotients, and has a separating family of tracial states (so in particular when $A$ is also simple and stably finite by the work of Blackadar, Handelman \cite{H:MJM,BH:JFA} and Haagerup \cite{H:QTrace}). 
\begin{corollary}
Let $A$ be a separable, unital and nuclear $\mathrm{C}^*$-algebra without elementary quotients, and with the property that for every positive element $x\in A$ there is a tracial state $\tau$ on $A$ with $\tau(x)\neq 0$. Then $C_0(0,1]\otimes A$ is quasidiagonal.
\end{corollary}
\begin{proof}
Suppose first that $\tau_{A}$ is an extremal tracial state on $A$ and take $B$ to be a UHF algebra in Proposition \ref{EmbeddingCone}, and let $\Psi:A\rightarrow B_\omega$ be the order zero map given there. The duality between order zero maps from $A$ and $^*$-homomorphisms from $C_0(0,1]\otimes A$ \cite[Corollary 4.1]{WZ:MJM}, gives a $^*$-homomorphism $\Pi:C_0(0,1]\otimes A\rightarrow B_\omega$ which has $\Pi(\id_{(0,1]}\otimes x)=\Psi(x)$.  Note that 
\begin{equation}\label{QD.1}
\tau_{B,\omega}(\Psi(1_A)^n)=1,\quad n\in\mathbb N
\end{equation} as $1_{B_\omega}-\Psi(1_A)\in J_B$. Then the measure on $(0,1]$ induced by $\tau_{B,\omega}(\Pi(\cdot\otimes 1_A))$ is supported at $1$ so that $\tau_{B,\omega}\circ\Pi$ vanishes on $C_0(0,1)\otimes A$, and hence factors through $A\cong (C_0(0,1]\otimes A)/(C_0(0,1)\otimes A)$.  Accordingly $\tau_{B,\omega}\circ\Pi=\delta_1\otimes\tau_A$.

 Given $t\in (0,1]$, follow the restriction $C_0(0,1]\rightarrow C_0(0,t]$ with a map $C_0(0,t]\rightarrow C_0(0,1]$ induced by rescaling to obtain a $^*$-homomorphism $\Pi_t:C_0(0,1]\otimes A\rightarrow B_\omega$ with $\tau_{B,\omega}\circ\Pi_t=\delta_t\otimes\tau_A$. Now if $\tau_{A}$ is faithful, the sequence of maps $(\widetilde{\Pi}_n)_{n=1}^\infty$ given by
\begin{equation}\label{QD.2}
\widetilde{\Pi}_n(f)=\Pi_{1/n}(x)\oplus\Pi_{2/n}(f)\oplus\dots\oplus \Pi_1(f):A\rightarrow B_\omega^{\oplus n}\cong (B^{\oplus n})_\omega
\end{equation}
can be used to confirm quasidiagonality of $C_0(0,1]\otimes A$ as asymptotically 
\begin{equation}
\frac{1}{n}\sum_{i=1}^n\tau_{B,\omega}(\Pi_{i/n}(f))\rightarrow\int_0^1\tau_A(f(t))dt,\quad f\in C_0((0,1],A)\cong C_0(0,1]\otimes A,
\end{equation}
where the integral on the right is performed with respect to Lebesgue measure and so is non-zero when $f$ is positive and non-zero, as $\tau_A$ was assumed faithful.

In general, note that for every non-zero positive element $f\in C_0((0,1],A)$, the hypothesis gives an extremal trace $\tau$ on $A$ such that $\int_0^1\tau(f(t))dt\neq 0$.  Thus sums of the maps in (\ref{QD.2}) taken over finite subsets of the extremal boundary of the tracial state space of $A$ can be used to witness quasidiagonality of $C_0(0,1]\otimes A$.
\end{proof}

\section{$2$-coloured approximately inner flips}\label{IFlip}

\noindent
The next lemma extracts the 2-coloured approximately inner flip for monotracial nuclear UHF-stable $\mathrm{C}^*$-algebras from the proof of \cite[Theorem 4.2]{MS:Preprint} in the form we need for the sequel.  We give the proof for completeness (note that our notation differs from that in \cite{MS:Preprint}, particularly for ultrapowers).
\begin{lemma}\label{FlipLemma}
Let $A$ be a simple, separable, unital, non-elementary and nuclear $\mathrm{C}^*$-algebra with unique tracial state $\tau_A$, and let $V$ be a UHF algebra.  Then there exist contractions $v^{(0)},v^{(1)}\in (A\otimes A\otimes V)_\omega$ with
\begin{equation}\label{FlipLem.1}
\sum_{i=0}^1v^{(i)}(a\otimes b\otimes 1_V)v^{(i)}{}^*=(b\otimes a\otimes 1_V),\quad a,b\in A,
\end{equation}
\begin{equation}\label{FlipLem.2}
v^{(0)}{}^*v^{(0)},\ v^{(1)}{}^*v^{(1)}\in (1_A\otimes 1_A\otimes V)_\omega\text{ and }\sum_{i=0}^1v^{(i)}{}^*v^{(i)}=1_A\otimes 1_A\otimes 1_V.
\end{equation}
\end{lemma}
\begin{proof}
As $A$ has unique trace its strong closure in the GNS-representation of this trace is an injective II$_1$ factor, and so has a $\|\cdot\|_{2,\tau}$-approximately inner flip by the implication (7)$\Rightarrow$(3) from Connes' Theorem \cite[Theorem 5.1]{C:Ann}).  In particular, using Kaplansky's density theorem, there exists a sequence $(u_n)_{n=1}^\infty$ of unitaries in $A\otimes A$ with
\begin{equation}\label{FlipLem.StrongFlip}
\lim_{n\rightarrow\omega}\|u_n x u_n^*- \sigma(x)\|_{2,\tau_{A\otimes A}}=0,\quad x\in A\otimes A,
\end{equation}
where $\sigma:A\otimes A\rightarrow A\otimes A$ is the flip automorphism of $A\otimes A$.

Write $B=M_{2} \otimes A\otimes A \otimes V$ and let $q:B_\omega\rightarrow B_\omega/J_{B,\tau_{B}}$  denote the quotient map. Consider the embedding $\pi:A\otimes A\rightarrow B_\omega\cong M_2((A\otimes A \otimes V)_\omega)$, given by 
\begin{equation}
\pi(x)=\begin{bmatrix}\iota(x) \otimes 1_{V_{\omega}}&0\\0&\iota\circ\sigma(x) \otimes 1_{V_{\omega}}\end{bmatrix},
\end{equation}
where $\iota:A\otimes A\hookrightarrow (A\otimes A)_\omega$ is the canonical unital inclusion.  Denote the partial isometry in $B_\omega$ represented by\begin{equation}
\left(\begin{bmatrix}0&0\\u_n \otimes 1_{V}&0\end{bmatrix}\right)_{n=1}^\infty
\end{equation}
by $u$, so that $q(u)\in q(B_\omega) \cap q(\pi(A\otimes A))'$ from (\ref{FlipLem.StrongFlip}). Then we have projections
\begin{equation}
e=\begin{bmatrix}1_{(A\otimes A \otimes V)_\omega}&0\\0&0\end{bmatrix},\quad f=\begin{bmatrix}0&0\\0&1_{(A\otimes A \otimes V)_\omega}\end{bmatrix}
\end{equation}
in $B_\omega\cap \pi(A\otimes A)'$ satisfying $e=u^*u$ and $f=uu^*$.  By construction, the partial isometry $q(u)$ witnesses Murray-von Neumann equivalence of $q(e)$ and $q(f)$ in $q(B_\omega) \cap q(\pi(A\otimes A))'$.  

The relative SI results of \cite[Section 3]{MS:Preprint} apply to the inclusion $\pi(A\otimes A)\subset B_\omega$. In particular every trace on $B_\omega \cap \pi(A\otimes A)'$ is of the form $\widetilde{\tau}\circ q$ for some trace $\widetilde{\tau}$ on $q(B_\omega) \cap q(\pi(A\otimes A))'$ by \cite[Proposition 3.3(i)]{MS:Preprint}. Thus $\tau(e)=\tau(f)$ for all traces $\tau\in T(B_\omega \cap \pi(A\otimes A)')$.
Further, as $B$ has strict comparison, \cite[Proposition 3.3(ii)]{MS:Preprint} applies to show that $B_\omega \cap \pi(A\otimes A)'$ has strict comparison of projections. Then, for each $n\in\mathbb N$, apply the $2$-coloured Cuntz-Pedersen theorem (\cite[Lemma 2.1]{MS:Preprint}) to obtain contractions $\widetilde{w}_n^{(0)},\widetilde{w}^{(1)}_n$ in $(B_\omega\cap \pi(A\otimes A)')\otimes M_n$, such that
\begin{equation}
\|\widetilde{w}^{(0)}_n{}^*\widetilde{w}_n^{(0)}+\widetilde{w}_n^{(1)}{}^*\widetilde{w}_n^{(1)}-e\otimes 1_n\|\leq \frac{4}{n},\quad \|\widetilde{w}_n^{(0)}\widetilde{w}_n^{(0)}{}^*+\widetilde{w}_n^{(1)}\widetilde{w}_n^{(1)}{}^*-f\otimes 1_n\|\leq \frac{4}{n}
\end{equation}
and
\begin{equation}
\mathrm{dist}(\widetilde{w}_n^{(0)}{}^*\widetilde{w}_n^{(0)},e\otimes M_n)\leq \frac{2}{n},\quad \mathrm{dist}(\widetilde{w}_n^{(1)}{}^*\widetilde{w}_n^{(1)},e\otimes M_n)\leq \frac{2}{n}.
\end{equation}

By regarding $M_n$ (for suitable $n$) as being unitally embedded in $V$, we have $\widetilde{w}^{(0)}_n,\widetilde{w}^{(1)}_n\in B_\omega \cap \pi(A\otimes A)'$.  A standard reindexing argument gives contractions $\widetilde{v}^{(0)},\widetilde{v}^{(1)}\in B_\omega \cap \pi(A\otimes A)'$ with 
\begin{equation}\label{LemB.1}
\sum_{i=0}^1\widetilde{v}^{(i)}{}^*\widetilde{v}^{(i)}=\begin{bmatrix}1_{(A\otimes A\otimes V)_\omega}&0\\0&0\end{bmatrix},\quad \sum_{i=0}^1\widetilde{v}^{(i)}\widetilde{v}^{(i)}{}^*=\begin{bmatrix}0&0\\0&1_{(A\otimes A\otimes V)_\omega}\end{bmatrix}
\end{equation}
and
\begin{equation}\label{LemB.2}
\widetilde{v}^{(0)}{}^*\widetilde{v}^{(0)},\ \widetilde{v}^{(1)}{}^*\widetilde{v}^{(1)}\in \begin{bmatrix}(1_A\otimes 1_A\otimes V)_\omega&0\\0&0\end{bmatrix}.
\end{equation}
From (\ref{LemB.1}), there are contractions $v^{(0)},v^{(1)}\in (A\otimes A\otimes V)_\omega$ such that
\begin{equation}
\widetilde{v}^{(0)}=\begin{bmatrix}0&0\\ v^{(0)}&0\end{bmatrix},\quad \widetilde{v}^{(1)}=\begin{bmatrix}0&0\\ v^{(1)}&0\end{bmatrix}.
\end{equation}
These are readily seen to satisfy (\ref{FlipLem.1}) and (\ref{FlipLem.2}).
\end{proof}

Since UHF algebras of infinite type are strongly self-absorbing, if a separable $\mathrm{C}^*$-algebra $A$ tensorially absorbs a UHF-algebra $U$ of infinite type, then there is an isomorphism $A\rightarrow A\otimes U$ which is approximately unitarily equivalent to the first factor embedding $\iota:A\rightarrow A\otimes U$ \cite[Theorem 7.2.2 and Remark 7.2.3]{R:Book}.  Applying this statement to $A\otimes A$, one obtains an isomorphism $\Theta:(A\otimes A\otimes U)_\omega\rightarrow (A\otimes A)_\omega$ with $\Theta(a\otimes b\otimes 1_U)=a\otimes b$ for $a,b\in A$. Following Proposition \ref{FlipLemma} by this isomorphism gives the next statement, the conclusion of which one might take as the definition that ``$A$ has $2$-coloured approximately inner flip''.
\begin{proposition}\label{FlipFollowup}
Let $A$ be a simple, separable, unital and nuclear $\mathrm{C}^*$-algebra with a unique tracial state which absorbs a UHF-algebra of infinite type.  Then there exist contractions $v^{(0)},v^{(1)}\in (A\otimes A)_\omega$ with $v^{(i)}{}^*v^{(i)}\in (A\otimes A)_\omega \cap (A\otimes A)'$ for $i=0,1$, $\sum_{i=0}^1v^{(i)}{}^*v^{(i)}=1_{A\otimes A}$, and
\begin{equation}\label{FlipFollowup.1}
\sum_{i=0}^1v^{(i)}(a\otimes b)v^{(i)}{}^*=b\otimes a,\quad a,b\in A.
\end{equation}
\end{proposition}

Having an approximately inner flip is a strong requirement to impose on a separable $\mathrm{C}^*$-algebra $A$, forcing $A$ to be simple and nuclear, and to have at most one tracial state \cite[Proposition 2.7, Proposition 2.8, Lemma 2.9]{ER:PJM}.  Essentially the same arguments apply to $\mathrm{C}^*$-algebras satisfying the conclusions of Proposition \ref{FlipFollowup}.  We sketch this below for completeness.  In particular this shows that in using Lemma \ref{FlipLemma} to prove Theorem \ref{Main}, the unique trace hypothesis really is essential.
\begin{proposition}[cf.~ \cite{ER:PJM}]
Let $A$ be a separable unital $\mathrm{C}^*$-algebra with $2$-coloured approximately inner flip, i.e., satisfying the conclusions of Proposition \ref{FlipFollowup}.  Then $A$ is simple, nuclear and has at most one tracial state.
\end{proposition}
\begin{proof}
For simplicity, if $I$ is a proper ideal in $A$, then $I\otimes A$ and $A\otimes I$ are distinct ideals in $A\otimes A$.  On the other hand if the flip automorphism on $A$ takes the form (\ref{FlipFollowup.1}), then it follows that $I\otimes A=A\otimes I$, a contradiction.  For the tracial statement, suppose $\tau_1$ and $\tau_2$ are traces on $A$, inducing a trace $(\tau_1\otimes\tau_2)_\omega$ on $(A\otimes A)_\omega$. Noting that $\sum_{i=0}^1v^{(i)}{}^*v^{(i)}=1_{(A\otimes A)_\omega}$, applying $(\tau_1\otimes\tau_2)_\omega$ to (\ref{FlipFollowup.1}) with $b=1_A$ shows that $\tau_1(a)=\tau_2(a)$ for $a\in A$.  Finally, for nuclearity, given (\ref{FlipFollowup.1}), lift each $v^{(i)}$ isometrically to a sequence $(v^{(i)}_n)_{n=1}^\infty$ of finite combinations of elementary tensors, $v^{(i)}_n=\sum_{j\in \Lambda_n}c^{(i)}_{j,n}\otimes d^{(i)}_{j,n}$.  Fix a state $\phi$ on $A$; then the maps
\begin{equation}
A\rightarrow A;\quad a\mapsto \sum_{i=0}^1\sum_{j,k\in \Lambda_n}\phi(c^{(i)}_{j,n}ac^{(i)}_{k,n}{}^*)d_{j,n}^{(i)}d_{k,n}^{(i)}{}^*
\end{equation}
are c.p., uniformly bounded, finite rank, and converge to $\id_A$ as $n\rightarrow\omega$ in point norm topology, so that $A$ is nuclear.
\end{proof}

\section{Tracially large purely positive elements in $\Z_\omega$}\label{Shuff}

\noindent
Our objective in this section is to establish the following uniqueness result for certain positive contractions of full spectrum which we shall apply when $A=B=\Z$.  

\begin{lemma}\label{ShuffleLemma}
Let $A$ and $B$ be simple, unital and exact $\mathrm{C}^*$-algebras with unique tracial states $\tau_{A}$ and $\tau_B$ respectively such that $A\otimes B$ has strict comparison and stable rank one. Let $h\in B$ be a positive contraction with spectrum $[0,1]$.  Given positive contractions $e^{(0)},e^{(1)}\in J_{A,\tau_{A}}\lhd A_\omega$, there is a unitary $u\in (A\otimes B)_\omega$ with $u((1_{A_\omega}-e^{(0)})\otimes h)u^*=(1_{A_\omega}-e^{(1)})\otimes h$.
\end{lemma}

Recall that $C_0(0,1]$ is the universal $\mathrm{C}^*$-algebra generated by a positive contraction, so positive contractions in a $\mathrm{C}^*$-algebra $D$ are in bijective correspondence with $^*$-homomorphisms $C_0(0,1]\rightarrow D$.  To prove Lemma \ref{ShuffleLemma}, we use Ciuperca and Elliott's classification of such $^*$-homomorphisms up to approximate unitary equivalence when $D$ has stable rank one via the Cuntz semigroup from \cite{CE:IMRN} (stated in the form we use as Theorem \ref{CuClassification} below; see also \cite{R:Adv}).

\begin{theorem}[{Ciuperca-Elliott \cite[Theorem 4]{CE:IMRN}}]\label{CuClassification}
Let $D$ be a unital $\mathrm{C}^*$-algebra with stable rank one, and for $i=0,1$, let $\rho^{(i)}:C_0(0,1]\rightarrow D$ be $^*$-homomorphisms.  Then $\rho^{(0)}$ and $\rho^{(1)}$ are approximately unitarily equivalent if and only if the induced maps $\Cu(\rho^{(i)}):\Cu(C_0(0,1])\rightarrow\Cu(D)$ are equal.
\end{theorem}

In the proof of Lemma \ref{ShuffleLemma} we use strict comparison in order to show that the $^*$-homomorphisms of Theorem \ref{CuClassification} do indeed induce the same map at the level of the Cuntz semigroup. 
Note that when $\tau$ is a faithful trace on $D$, and $h\in D_+$ is a positive contraction with full spectrum, then $d_\tau((h-\eps)_+)<d_\tau(h)$ for all $\eps>0$.  

\begin{proof}[Proof of Lemma \ref{ShuffleLemma}.]
By transitivity, it suffices to prove the lemma with $e^{(1)}=0$.  For $i=0,1$, write $\rho^{(i)}$ for the induced $^*$-homomorphism $C_0(0,1]\rightarrow (A\otimes  B)_\omega$, i.e.,
\begin{equation}
\rho^{(0)}(\mathrm{id}_{(0,1]}) = (1_{A_{\omega}} - e^{(0)}) \otimes h
\end{equation}
and
\begin{equation}
\rho^{(1)}(\mathrm{id}_{(0,1]}) = 1_{A_{\omega}} \otimes h.
\end{equation}
We shall show that
\begin{equation}\label{Shuffle.Key}
\langle \rho^{(0)}(f)\rangle =\langle \rho^{(1)}(f)\rangle\text{ in }\Cu((A\otimes B)_\omega),\quad f\in C_0(0,1].
\end{equation}
Since $\Cu(C_0(0,1])$ is isomorphic to $\mathrm{lsc}((0,1],\{0,1,2,\dots,\infty\})$, the lower semicontinuous functions from $(0,1]$ into the extended natural numbers (a result of Coward which can be found as \cite[Theorem 10.1]{CE:IMRN}), it is generated (under the operations of addition and taking suprema) by the equivalence classes of positive elements in $C_0(0,1]$.  Thus, once we have established (\ref{Shuffle.Key}) it follows that $\Cu(\rho^{(0)})=\Cu(\rho^{(1)})$, and hence $\rho^{(0)}$ and $\rho^{(1)}$ are approximately unitarily equivalent by Theorem \ref{CuClassification}, as $(A\otimes B)_\omega$ inherits stable rank one from $A\otimes B$ by Lemma \ref{UnitaryPolar}.  Then either a reindexing argument, or the ``$\eps$-test'' \cite[Lemma A.1]{K:Abel} can be used to show that $\rho^{(0)}$ and $\rho^{(1)}$ are actually unitarily equivalent.

We now establish (\ref{Shuffle.Key}).  Write $\tau_{A\otimes B}=\tau_A\otimes \tau_B$. Lift $e^{(0)}$ to a sequence $(e_n)_{n=1}^\infty$ of positive contractions in $A$, which satisfy $\lim_{n\rightarrow\omega}\tau_A(e_n)=0$.  Write $\mu_n$ for the measure on $(0,1]$ induced by 
\begin{equation}
\mu_n(f)=\tau_{A\otimes B}(f((1_A-e_n)\otimes h)), \quad f\in C_0(0,1]
\end{equation}
and $\mu$ for the measure induced by 
\begin{equation}
\mu(f)=\tau_{A\otimes B}(f(1_A\otimes h)),\quad f\in C_0(0,1].
\end{equation}
As $e^{(0)}\in J_A$, we have
\begin{equation}
\lim_{n\rightarrow\omega}\tau_{A\otimes B}((1_A-e_n)^m\otimes h^m)= \tau_{A\otimes B}(1_A\otimes h^m),\quad m\in\mathbb N
\end{equation}
and so, by approximating $f\in C_0(0,1]$ uniformly by polynomial functions, 
\begin{equation}
\lim_{n\rightarrow\omega}\tau_{A\otimes B}(f((1_A-e_n)\otimes h))=\tau_{A\otimes B}(f(1_A\otimes h)),\quad f\in C_0(0,1].
\end{equation}
Given an open set $U\subset (0,1]$, take a sequence $(g_m)_{m=1}^\infty$ of continuous functions in $C_0(0,1]$ with $0\leq g_m\scaleobj{0.75}{\nearrow} \chi_U$, where $\chi_U$ is the indicator function of $U$.  Then
$$
\mu_n(\chi_{U})\geq \mu_n(g_m)\stackrel{n\rightarrow\omega}\rightarrow \mu(g_m),
$$
so that 
\begin{equation}\label{Shuffle.1}
\lim_{n\rightarrow\omega}\mu_n(\chi_{U})\geq \mu(g_m)\stackrel{m\rightarrow\infty}{\rightarrow}\mu(U),
\end{equation}
by the monotone convergence theorem.  Consequently, for closed subsets $F\subseteq(0,1]$, we have 
\begin{equation}\label{Shuffle.Closed}
\lim_{n\rightarrow\omega}\mu_n(F)\leq\mu(F).
\end{equation}

Now, for $f\in C_0(0,1]$, 
\begin{equation}\label{Shuffle.2}
d_{\tau_{A\otimes B}}(f((1_A-e_n)\otimes h))=\mu_n(\supp(f)).
\end{equation}
Thus (\ref{Shuffle.1}) and (\ref{Shuffle.2}) combine to show
\begin{equation}\label{Shuffle.3}
\lim_{n\rightarrow\omega}d_{\tau_{A\otimes B}}(f((1_A-e_n)\otimes h))\geq d_{\tau_{A\otimes B}}(f(1_A\otimes h)),\quad f\in C_0(0,1]_+.
\end{equation}

Fix $f\in C_0(0,1]_+$ with $\|f\|=1$.  For notational purposes write $\widetilde{h}=f(1_A\otimes h)$, and $\widetilde{h}_n=f((1_A-e_n)\otimes h)$.  As $h$ has full spectrum, so too has $\widetilde{h}$. Since $A$ and $B$ (and hence $A\otimes B$, since we work with the spatial tensor norm) are simple, it follows that for a fixed, sufficiently small $\eps>0$, we have
\begin{equation}\label{Shuffle.3a}
d_{\tau_{A\otimes B}}((\widetilde{h}-\eps)_+)<d_{\tau_{A\otimes B}}((\widetilde{h}-\eps/2)_+)<d_{\tau_{A\otimes B}}(\widetilde{h}).
\end{equation}
Define 
\begin{equation}r_{\eps/2}(t)=\begin{cases}2t/\eps&0\leq t\leq\eps/2;\\1,&t\geq \eps/2.\end{cases}
\end{equation}
Then $r_{\eps/2}(t) (t-\eps/2)_+=(t-\eps/2)_+$.  By (\ref{Shuffle.3}) and (\ref{Shuffle.3a}), we have
\begin{equation}\label{Shuffle.RepeatFromHere}
\lim_{n\rightarrow\omega}d_{\tau_{A\otimes B}}((\widetilde{h}_n-\eps/2)_+)\geq d_{\tau_{A\otimes B}}((\widetilde{h}-\eps/2)_+)>d_{\tau_{A\otimes B}}((\widetilde{h}-\eps)_+).
\end{equation}
Thus, the set $\Lambda=\{n\in\mathbb N:d_{\tau_{A \otimes B}}((\widetilde{h}_n-\eps/2)_+)>d_{\tau_{A \otimes A}}((\widetilde{h}-\eps)_+)\}$ lies in $\omega$, and for $n\in \Lambda$, $(\widetilde{h}-\eps)_+\precsim (\widetilde{h}_n-\eps/2)_+$ by strict comparison (exactness of $A$ and $B$, and hence of $A\otimes B$, plays a role here ensuring via Haagerup's work \cite{H:QTrace} that the unique trace is the only $2$-quasitrace on $A\otimes B$).  For $n\in \Lambda$, find $v_n\in A\otimes B$ (it is easy to see that $v_n$ can be taken in $A\otimes B$ regarded as a subalgebra of $A\otimes B\otimes\mathcal K$) with $\|v_n(\widetilde{h}_n-\eps/2)_+v_n^*-(\widetilde{h}-\eps)_+\|<1/n$ and for $n\notin \Lambda$, define $v_n$ arbitrarily with $\|v_n\|\leq 1$.  Then
\begin{equation}
\lim_{n\rightarrow\omega}\|v_n(\widetilde{h}_n-\eps/2)_+^{1/2}r_{\eps/2}(\widetilde{h}_n)(\widetilde{h}_n-\eps/2)_+^{1/2}v_n^*-(\widetilde{h}-\eps)_+\|=0.
\end{equation}
As the sequence $(v_n(\widetilde{h}_n-\eps/2)_+^{1/2})_{n=1}^\infty$ is uniformly bounded, this witnesses 
\begin{equation}
(f(1_A\otimes h)-\eps)_+\precsim r_{\eps/2}(f((1_{A_\omega}-e^{(0)})\otimes h))\sim f((1_{A_\omega}-e^{(0)})\otimes h)
\end{equation}
in $(A\otimes B)_\omega$. Since $\eps>0$ can be taken arbitrarily small, this shows that
\begin{equation}\label{Shuffle.ToHere}
\langle f(1_A\otimes h)\rangle\leq \langle f((1_{A_\omega}-e^{(0)})\otimes h)\rangle\text{ in }\Cu((A\otimes B)_\omega).
\end{equation}

The reverse inequality is similar, working with closed supports.  Fix an $\eps>0$ small enough that $\supp((f-\eps/2)_+)\supsetneq\overline{\supp((f-\eps)_+)}$, and 
\begin{equation}
\mu\big(\overline{\supp((f-\eps)_+})\big)<d_{\tau_{A\otimes B}}((\widetilde{h}-\eps/2)_+).
\end{equation}
By (\ref{Shuffle.Closed}), 
\begin{equation}
\lim_{n\rightarrow\omega}\mu_n\big(\overline{\supp((f-\eps)_+)}\big)<d_{\tau_{A\otimes B}}((\widetilde{h}-\eps/2)_+),
\end{equation}
so that 
\begin{equation}
\lim_{n\rightarrow\omega}d_{\tau_{A\otimes B}}((\widetilde{h}_n-\eps)_+)<d_{\tau_{A\otimes B}}((\widetilde{h}-\eps/2)_+).
\end{equation}
Then, arguing in just the same way that (\ref{Shuffle.ToHere}) is obtained from (\ref{Shuffle.RepeatFromHere}), we have
\begin{equation}
\langle f(1_A\otimes h)\rangle\geq \langle f((1_{A_\omega}-e^{(0)})\otimes h)\rangle\text{ in }\Cu((A\otimes B)_\omega).
\end{equation}
Combining this with (\ref{Shuffle.ToHere}) establishes (\ref{Shuffle.Key}), and so completes the proof.
\end{proof}

Note that when $B$ is the Jiang-Su algebra, then when $A$ is simple, unital, exact and stably finite, $A\otimes\Z$ has strict comparison and stable rank one by \cite{R:IJM}, leading to the following version of the lemma, in which we can also replace $\Z$ by any UHF-algebra.
\begin{lemma}
Let $A$ be a simple, unital and exact $\mathrm{C}^*$-algebra with unique tracial state $\tau_{A}$. Let $h\in \Z$ be a positive contraction with spectrum $[0,1]$.  For any pair of positive contractions $e^{(0)},e^{(1)}\in J_{A,\tau_{A}}\lhd A_\omega$ there is a unitary $u\in (A\otimes \Z)_\omega$ with $u((1_{A_\omega}-e^{(0)})\otimes h)u^*=(1_{A_\omega}-e^{(1)})\otimes h$.
\end{lemma}

\section{$\mathcal Z$-stability and nuclear dimension}\label{LastSect}

\noindent
In this section we establish the main result of the paper: Theorem \ref{Main}.  Recall that the Jiang-Su algebra has strict comparison and stable rank one (\cite{R:IJM}, which shows the same for all simple, separable, unital, exact and stably finite $\Z$-stable $\mathrm{C}^*$-algebras). Thus Lemma \ref{ShuffleLemma} can be applied with $A=B=\Z$.

From the original inductive limit construction one can see that the Jiang-Su algebra is almost divisible (again see \cite{R:IJM} for the more general statement in the presence of $\Z$-stability); in particular given $\eps>0$ and $k\in\mathbb N$, there exists a c.p.c.~ order zero map $\psi:M_k\rightarrow\Z$ with $\tau_\Z(\psi(1_k))>1-\eps$. (Note too that these maps can be obtained using the maps $\Phi$ from Proposition \ref{NewProp1.2}).  We use this last fact in the first of two technical lemmas, which allow us to adjoin the UHF tensor factor required to employ Lemma \ref{FlipLemma} and factorise this through a sum of two order zero maps back into $\Z$.  The two order zero maps used here double the number of colours used in the final approximation: two colours arise from the $2$-coloured approximately inner flip, and another two from Lemma \ref{2OrderZeroMaps}, leading to a $4$-colour approximation in Theorem \ref{Main}.

\begin{lemma}\label{LargeEmbed}
Let $U$ be a UHF algebra. Then there exists a sequence $(\widetilde{\psi}_n)_{n=1}^\infty$ of c.p.c.~ order zero maps $U\rightarrow \Z$ inducing an order zero map $\widetilde{\Psi}:U_\omega\rightarrow \Z_\omega$ such that for any positive contraction $e\in J_{U, \tau_{U}}$, we have $1_{\Z_\omega}-\widetilde{\Psi}(1_{U_\omega}-e)\in J_{\Z, \tau_{\Z}}$.
\end{lemma}
\begin{proof}
Let $p$ and $q$ be supernatural numbers which are relatively prime and such that $U$ embeds unitally into $M_{p} \otimes M_{q}$. As pointed out in \cite[Proposition 3.3]{RW:Crelle}, the generalised dimension drop interval $Z_{p,q}$ embeds unitally into $\Z$. Careful inspection of the proof (which follows that of \cite[Proposition 2.2]{R:IJM}) shows that this embedding can be chosen so that the trace on $\Z$ induces the trace on $Z_{p,q}$ associated to the Lebesgue integral on $[0,1]$. Note that $Z_{p,q} \subset C([0,1],M_{p} \otimes M_{q})$, and that there is an order zero map $M_{p} \otimes M_{q} \rightarrow Z_{p,q}$ such that the Lebesgue trace on $Z_{p,q}$ is at most $1-1/n$ on the image of the unit of $M_{p} \otimes M_{q}$. (Embed the cone $C_{0}(0,1] \otimes M_{p} \otimes M_{q}$ into $C[0,1] \otimes M_{p} \otimes M_{q} \cong C([0,1],M_{p} \otimes M_{q})$ in such a way that $\mathrm{id}_{(0,1]}$ maps to a function which vanishes at the endpoints and is constant $1$ on the interval $[1/2n, 1-1/2n]$.) We may now compose these maps to obtain a sequence $(\widetilde{\psi}_n)_{n=1}^\infty$ of order zero maps $U\rightarrow\Z$ with $\tau_\Z(\widetilde{\psi}_n(1_U))>1-1/n$ for each $n\in\mathbb N$. Define $\widetilde{\Psi}$ to be the c.p.c.~ order zero map $U_\omega\rightarrow \Z_\omega$ induced by the sequence $(\widetilde{\psi}_n)_{n=1}^\infty$.  Given a positive contraction $e\in J_{U,\tau_{U}}$, lift $e$ to a sequence $(e_n)_{n=1}^\infty$ of positive contractions with $\lim_{n\rightarrow\omega}\tau_{U}(e_n)=0$.  Since $\tau_\Z\circ \widetilde{\psi}_n$ is a trace on $U$ (by \cite[Corollary 4.4]{WZ:MJM}), we have $\tau_\Z(\widetilde{\psi}_n(e_n))=\tau_U(e_n)\tau_\Z(\widetilde{\psi}_n(1_U))$. Then
\begin{align}
\tau_\Z(1_\Z-\widetilde{\psi}_n(1_U-e_n))&=1-\tau_\Z(\widetilde{\psi}_n(1_U))+\tau_\Z(\widetilde{\psi}_n(e_n))\nonumber\\
&=1+\tau_\Z(\widetilde{\psi}_n(1))(\tau_U(e_n)-1)\stackrel{n\rightarrow\omega}{\rightarrow} 1+1(0-1)=0.
\end{align}
Thus $1_{\Z_\omega}-\widetilde{\Psi}(1_{U_\omega}-e)\in J_{\Z,\tau_{\Z}}$, as claimed.
\end{proof}

\begin{lemma}\label{2OrderZeroMaps}
Let $U$ be a UHF algebra, and let $e\in J_{U,\tau_{U}}$ be a positive contraction.  Then there exist two sequences $(\lambda^{(0)}_n)_{n=1}^\infty$ and $(\lambda^{(1)}_n)_{n=1}^\infty$ of c.p.c.~ order zero maps $U\rightarrow\Z$ inducing c.p.c.~ order zero maps $\Lambda^{(0)},\Lambda^{(1)}:U_\omega\rightarrow\Z_\omega$  such that $\sum_{j=0}^1\Lambda^{(j)}(1_{U_\omega}-e)=1_{\Z_{\omega}}$.  
\end{lemma}
\begin{proof}
Fix a positive contraction $d^{(0)}\in \Z$ with spectrum $[0,1]$, and write $d^{(1)}=1-d^{(0)}$ which is also a positive contraction with spectrum $[0,1]$. By Lemma \ref{LargeEmbed}, there is a sequence $(\widetilde{\lambda}_n)_{n=1}^\infty$ of c.p.c.~ order zero maps $U\rightarrow \mathcal{Z}$ inducing a c.p.c.~ order zero map $\widetilde{\Lambda}:U_\omega\rightarrow\Z_\omega$ such that $1_{\Z_\omega}-\widetilde{\Lambda}(1_{U_\omega}-e)\in J_{\Z, \tau_{\Z}}$.  Then consider the c.p.c.~ order zero maps $\widetilde{\lambda}_n^{(j)}:U \rightarrow \Z\otimes \Z$ given by $\widetilde{\lambda}_n^{(j)}(x)=\widetilde{\lambda}_n(x)\otimes d^{(j)}$, which induce c.p.c.~ order zero maps $\widetilde{\Lambda}^{(j)}:U_\omega\rightarrow\Z_\omega \otimes \Z \subset (\Z \otimes \Z)_{\omega}$ with $\widetilde{\Lambda}^{(j)}(1_{U_\omega}-e)=(\widetilde{\Lambda}(1_{U_\omega}-e)\otimes d^{(j)})$.  By Lemma \ref{ShuffleLemma}, there exist unitaries $u^{(j)}\in (\Z\otimes \Z)_\omega$, with $u^{(j)}(\widetilde{\Lambda}^{(j)}(1_{U_\omega}-e))u^{(j)}{}^*=(1_{\Z_\omega}\otimes d^{(j)})$, which we can lift to a sequence $(u^{(j)}_n)_{n=1}^\infty$ of unitaries in $\Z\otimes\Z$.   Define c.p.c.~ order zero maps $\lambda_n^{(j)}=\mathrm{ad}(u_n^{(j)})\circ\widetilde{\lambda}_n^{(j)}:U\rightarrow \Z\otimes\Z$, with induced map $\Lambda^{(j)}$ so that $\sum_{j=0}^1\Lambda^{(j)}(1_{U_\omega}-e)=\sum_{j=0}^1(1_{\Z_\omega}\otimes d^{(j)})=1_{\Z_\omega}\otimes 1_\Z = 1_{(\Z \otimes \Z)_{\omega}}$.    Identifying $\Z\otimes\Z$ with $\Z$, the lemma is established.
\end{proof}

The next lemma packages a factorisation for the order zero maps obtained from the $2$-coloured approximately inner flip (cf.\ Proposition \ref{FlipFollowup}). 

\begin{lemma}\label{Factorise}
Let $A$ and $B$ be separable unital $\mathrm{C}^*$-algebras with $B$ simple and stably finite and let $U$ and $V$ be UHF algebras. Given a contraction $v\in (B\otimes A\otimes V)_\omega$ with $v^*v\in (1_B\otimes 1_A\otimes V)_\omega$, let $(\psi_n)_{n=1}^\infty$ be a sequence of c.p.c.~ maps $A\rightarrow U$, such that the c.p.c.~ map $\Psi:(B\otimes A\otimes V)_\omega\rightarrow (B\otimes U\otimes V)_\omega$ induced by the sequence $(\id_B\otimes \psi_n\otimes\id_V)_{n=1}^\infty$ is order zero on $\mathrm{C}^*(B\otimes A\otimes V,v)\subset (B\otimes A\otimes V)_\omega$.  Then the c.p.c.~ map $\Theta:A\rightarrow (B\otimes U\otimes V)_\omega$ given by $\Theta(a)=\Psi(v(1_B\otimes a\otimes 1_V)v^*)$ is order zero and can be factorised as follows: there are finite dimensional $\mathrm{C}^{*}$-algebras $G_n$, c.p.c.~ maps $\phi_n:A\rightarrow G_n$, and $^*$-homomorphisms $\theta_n:G_n\rightarrow B\otimes U\otimes V$ such that $\Theta(a)$ is represented by $(\theta_n(\phi_n(a)))_{n=1}^\infty$ for $a\in A$.
\end{lemma}
\begin{proof}
Since $v^*v\in (1_B\otimes 1_A\otimes V)_\omega$, the map $a\mapsto v(1_B\otimes a\otimes 1_V)v^*$ is order zero, and hence so too is $\Theta$.

By following each $\psi_n$ by a suitable conditional expectation from $U$ to a finite dimensional $\mathrm{C}^*$-subalgebra, we can assume that $\mathrm{C}^*(\psi_n(A))=E_n \subset U$ is finite dimensional for each $n$ without changing the values of the induced map $\Psi$ on the separable $\mathrm{C}^*$-algebra $X=\mathrm{C}^*(B\otimes A\otimes V,v)$.  Then Lemma \ref{SequenceStructureTheorem} gives a sequence of c.p.c.~ maps $(\pi_n)_{n=1}^\infty$, $\pi_n:A\rightarrow E_n$ such that $(\id_B\otimes\pi_n\otimes\id_V)_{n=1}^\infty$ induces a c.p.c.~ order zero map $\Pi:X\rightarrow (B\otimes U\otimes V)_\omega$ satisfying
\begin{equation}
\Psi(x)=\Pi(x)\Psi(1_{(B\otimes A\otimes V)_\omega})=\Psi(1_{(B\otimes A\otimes V)_\omega})\Pi(x),\quad x\in X.
\end{equation}
In particular
\begin{align}
\Theta(a)&=\Psi(v(1_B\otimes a\otimes 1_V)v^*)=\Pi(v(1_B\otimes a\otimes 1_V)v^*)\Psi(1_{(B\otimes A\otimes V)_\omega})\nonumber\\&=\Pi^{1/3}(v)\Psi(1)^{1/2}\Pi^{1/3}(1_B\otimes a\otimes 1_V)\Psi(1)^{1/2}\Pi^{1/3}(v^*),\quad a\in A.\label{Factorisation.1}
\end{align}

Write $s=\Pi^{1/3}(v)$.  By Lemma \ref{SequenceFunctionalCalculus}, $\Pi^{1/3}$ is induced by a sequence of c.p.c.~ maps of the form $(\id_B\otimes\widetilde{\pi}_n\otimes\id_V)_{n=1}^\infty$, where $\widetilde{\pi}_n:A\rightarrow E_n$.  Since $v^*v\in(1_B\otimes 1_A\otimes V)_\omega$,  we can lift $|s|=(\Pi^{1/3}(v^{*}) \Pi^{1/3}(v))^{1/2} = \Pi^{1/3}((v^*v)^{1/2})$ to a sequence $(t_n)_{n=1}^\infty$ of positive contractions in $1_B\otimes E_n\otimes V$.  Since $V$ is AF, by making a small perturbation of these lifts, we can assume that there are finite dimensional $\mathrm{C}^*$-subalgebras $F_n\subseteq V$ with $1_V\in F_n$ such that  $t_n\in 1_B\otimes E_n\otimes F_n$ for all $n$.  As $B$ is simple and stably finite,  $B\otimes U\otimes V$ has stable rank one by \cite[Theorem 5.7]{R:JFA}, so Lemma \ref{UnitaryPolar} provides a unitary $u\in (B\otimes U\otimes V)_\omega$ with $s=u|s|$. Lift $u$ to a sequence of unitaries $(u_n)_{n=1}^\infty$ in $B\otimes U\otimes V$, then $(u_nt_n)_{n=1}^\infty$ is a lift of $s$.  

Define $G_n=1_B\otimes E_n\otimes F_n$ and c.p.c.~ maps $\phi_n:A\rightarrow G_n$ by 
\begin{equation}\label{Factorisation.2}
\phi_n(a)=t_n(1_B\otimes \psi_n(1_A)^{1/2}\widetilde{\pi}_n(a)\psi_n(1_A)^{1/2}\otimes 1_V)t_n,\quad a\in A,
\end{equation}
noting that the choices have been made so that $\phi_n(A)\subseteq G_n$.  Let $\theta_n:G_n\rightarrow B\otimes U\otimes V$ be the $^*$-homomorphism $\theta_n=\mathrm{ad}(u_n)$. As $(u_nt_n)_{n=1}^\infty$ is a lift of $\Pi^{1/3}(v)$, (\ref{Factorisation.1}) and (\ref{Factorisation.2}) show that for every $a\in A$, $(\theta_n(\phi_n(a)))_{n=1}^\infty$ is a lift of $\Theta(a)$, as required.
\end{proof}

All the pieces are now in place to prove our main result.  As explained in \cite[Propositions 2.4 and 2.5]{TW:APDE}, when $A$ is a $\Z$-stable $\mathrm{C}^*$-algebra, $\dimnuc(A)$ can be computed by only approximating the first factor embedding $\iota:A\hookrightarrow (A\otimes\Z)_\omega$; this is  used at the end of the proof.

\begin{proof}[Proof of Theorem \ref{Main}]
Fix UHF algebras $U$ and $V$. Since $A$ has unique tracial state, we can apply Lemma \ref{FlipLemma} to find contractions $v^{(0)},v^{(1)}\in (A\otimes A\otimes V)_\omega$ satisfying (\ref{FlipLem.1}) and (\ref{FlipLem.2}).   By Proposition \ref{EmbeddingCone} we can find an order zero map $\widetilde{\Psi}_2:A\rightarrow U_\omega$ with $1_{U_\omega}-\widetilde{\Psi}_{2}(1_A)\in J_{U,\tau_{U}}$, which, by the Choi-Effros lifting theorem of \cite{CE:Ann}, is induced by a sequence $(\widetilde{\psi}_n)_{n=1}^\infty$ of c.p.c.~ maps $A\rightarrow U$.  Letting $(x^{(l)})_{l=1}^\infty$ be a countable dense sequence in $\mathrm{C}^*(A\otimes A\otimes V,v^{(0)},v^{(1)}) \subset (A \otimes A \otimes V)_{\omega}$, with fixed isometric lifts $(x_n^{(l)})_{n=1}^\infty$, we can find a sequence $(m_n)_{n=1}^\infty$ of natural numbers such that 
\begin{align}
&\|(\id_A\otimes\widetilde{\psi}_{m_n}\otimes\id_V)(x^{(k)}_nx^{(l)}_n)(\id_A\otimes\widetilde{\psi}_{m_n}\otimes\id_V)(1_A\otimes 1_A\otimes 1_V)\\
&\quad-(\id_A\otimes\widetilde{\psi}_{m_n}\otimes\id_V)(x^{(k)}_n)(\id_A\otimes\widetilde{\psi}_{m_n}\otimes\id_V)(x^{(l)}_n)\|\leq\frac{1}{n},\quad 1\leq k,l\leq n,\nonumber
\end{align}
and $\tau_{U}(\psi_{m_n}(1_A))\geq 1- 1/n$ for each $n\in\mathbb N$.  Let $\psi_n=\widetilde{\psi}_{m_n}$ and let $\Psi:(A\otimes A\otimes V)_\omega\rightarrow (A\otimes U\otimes V)_\omega$ be the c.p.c.~ map induced by the sequence $(\id_A\otimes \psi_n\otimes\id_V)_{n=1}^\infty$ which is constructed to be order zero on $\mathrm{C}^*(A\otimes A\otimes V,v^{(0)},v^{(1)})$. Write $\Psi_2:A\rightarrow U_\omega$ for the c.p.c.~ order zero map induced by $(\psi_n)_{n=1}^\infty$ and note that $1_{U_\omega}-\Psi_2(1_A)\in J_{U,\tau_{U}}$.

Embedding $U_\omega$ into $(U\otimes V)_\omega$ as $U_\omega \otimes 1_V$, we have
\begin{equation}
1_{(U\otimes V)_\omega}-\Psi_2(1_A)\otimes 1_V=(1_{U_\omega}-\Psi_2(1_A))\otimes 1_V\in J_{U\otimes V, \tau_{U} \otimes \tau_{V}}.
\end{equation}
As $U\otimes V$ is UHF, Lemma \ref{2OrderZeroMaps} provides two sequences $(\widetilde{\lambda}^{(j)}_n)_{n=1}^\infty$, $j=0,1$, of c.p.c.~ order zero maps $U\otimes V\rightarrow\Z$ such that the induced order zero maps $\widetilde{\Lambda}^{(j)}:(U\otimes V)_\omega\rightarrow\Z_\omega$ satisfy $\sum_{j=0}^1\widetilde{\Lambda}^{(j)}(\Psi_2(1_A)\otimes 1_V)=1_{\Z_\omega}$.  For each $n$ and $j$, let $\lambda_n^{(j)}=\id_A\otimes\widetilde{\lambda}^{(j)}_n:A\otimes U\otimes V\rightarrow A\otimes\Z$, which is c.p.c.~ and order zero, and so the two sequences $(\lambda_n^{(j)})_{n=1}^\infty$ induce c.p.c.~ order zero maps $\Lambda^{(j)}:(A\otimes U\otimes V)_\omega\rightarrow (A\otimes \Z)_\omega$ with 
\begin{equation}\label{Main.1}
\sum_{j=0}^1\Lambda^{(j)}(a\otimes\Psi_2(1_A)\otimes 1_V)=a\otimes 1_{\Z_\omega},\quad a\in A.
\end{equation}
Let $\Theta^{(i)}:A\rightarrow (A\otimes U\otimes V)_\omega$, $i=0,1$, be the c.p.c.~ order zero map given by $\Theta^{(i)}(a)=\Psi(v^{(i)}(1_A\otimes a\otimes 1_V)v^{(i)}{}^*)$. Hence, for $a\in A$,
\begin{align}
\sum_{i,j=0}^1\Lambda^{(j)}(\Theta^{(i)}(a))&=\sum_{i,j=0}^1\Lambda^{(j)}(\Psi(v^{(i)}(1_A\otimes a\otimes 1_V)v^{(i)}{}^*))\nonumber\\
&=\sum_{j=0}^1\Lambda^{(j)}(\Psi(a\otimes 1_A\otimes 1_V))\nonumber\\
&=\sum_{j=0}^1\Lambda^{(j)}(a\otimes {\Psi}_2(1_A)\otimes 1_V)=a\otimes 1_{\Z},\label{Main.2}
\end{align}
using (\ref{FlipLem.1}) for the second equality, and (\ref{Main.1}) for the last.

We can factorise the order zero map $\Theta^{(i)}$ by Lemma \ref{Factorise} (with $A$ in place of $B$), so there are finite dimensional $\mathrm{C}^{*}$-algebras $G^{(i)}_n$, c.p.c.~ maps $\phi_n^{(i)}:A\rightarrow G_n^{(i)}$ and $^*$-homomorphisms $\theta_n^{(i)}:G_n^{(i)}\rightarrow (A\otimes U\otimes V)$ such that for each $a\in A$, $\Theta^{(i)}(a)$ is represented by the sequence $(\theta_n^{(i)}(\phi_n^{(i)}(a)))_{n=1}^\infty$. This shows that the first factor embedding $A \rightarrow A \otimes \Z$ has nuclear dimension at most $3$ (in the sense of \cite[Definition 2.2]{TW:APDE}), from which it follows that $\dimnuc(A)\leq 3$ by \cite[Propositions 2.5 and 2.6]{TW:APDE}. For completeness, we give the details here. As $A$ is $\Z$-stable, and $\Z$ is strongly self-absorbing, there is a sequence $(\sigma_n)_{n=1}^\infty$ of $^*$-isomorphisms $\sigma_n:A\otimes\Z\rightarrow A$ such that $\sigma_n(a\otimes 1_\Z)\rightarrow a$ for all $a\in A$ (see for example \cite[Theorem 2.2]{TW:TAMS}).  Then the diagram
\begin{equation}\label{Main.3}
\xymatrix{A\ar[dr]_{\bigoplus_{i,j=0}^1\phi_n^{(i)}\quad}\ar[rr]^{\id_A}&&A\\&{\bigoplus_{i,j=0}^1G^{(i)}_n}\ar[ur]_{\quad \sum_{i,j=0}^1\sigma_n\circ\lambda^{(j)}_n\circ\theta_n^{(i)}}}
\end{equation}
is point-norm asymptotically commutative as $n\rightarrow\omega$, showing $\dimnuc(A)\leq 3$.
\end{proof}

When $A$ has approximately inner flip we can use this in place of Lemma \ref{FlipLemma} and reduce the bound on the nuclear dimension.  In the proof below the factor $V$ is unnecessary and is included only so that we can quote our earlier results without notational modification; a formally simpler proof can be given just working with $A\otimes A$.

\begin{theorem}\label{ApproxFlipDim1}
Let $A$ be a separable, unital and stably finite $\mathrm{C}^*$-algebra, which has approximately inner flip and is $\Z$-stable.  Then $\dimnuc(A)\leq 1$. \end{theorem}
\begin{proof}
As $A$ has approximately inner flip it is simple and nuclear \cite{ER:PJM}.  Let $(u_n)_{n=1}^\infty$ be a sequence of unitaries in $A\otimes A$ witnessing the approximately inner flip. We can then follow the proof of Theorem \ref{Main} with $v_n^{(0)}=u_n\otimes 1_V$ and $v_n^{(1)}=0$, to obtain a factorisation (\ref{Main.3}) consisting of two maps (the maps with $i=1$ in (\ref{Main.2}) vanish), so that $\dimnuc(A)\leq 1$.
\end{proof}

In particular the previous result applies to the class of strongly self-absorbing algebras formalised in \cite{TW:TAMS}.
\begin{corollary}
Let $A$ be a stably finite strongly self-absorbing $\mathrm{C}^*$-algebra. Then $\dimnuc(A)\leq 1$.
\end{corollary}
\begin{proof}
Strongly self-absorbing $\mathrm{C}^*$-algebras are separable by definition \cite{TW:TAMS}, have approximately inner flip essentially by definition \cite[Proposition 1.5]{TW:TAMS}, and are $\Z$-stable by \cite{W:JNCG}. Thus Theorem \ref{ApproxFlipDim1} applies.
\end{proof}

\subsection*{Acknowledgments}

Part of the research for this paper was undertaken at the Oberwolfach
workshop \emph{$\mathrm{C}^*$-Algebren} and at the conference \emph{Classifying
Structures for Operator Algebras and Dynamical Systems} at the
University of Aberystwyth. The authors thank the organisers and
funding bodies of these meetings for the conducive research
environments. The paper was completed during a visit of the first and
second named authors to M\"u{}nster. These authors thank the third
named author, and the Mathematics Institute at M\"u{}nster, for their
hospitality.
Finally, all authors thank the referees for their helpful comments and suggestions.


\begin{thebibliography}{10}

\bibitem{APT:Contemp}
P.~Ara, F.~Perera, and A.~S. Toms.
\newblock {$K$}-theory for operator algebras. {C}lassification of
  {$\mathrm{C}^*$}-algebras.
\newblock In {\em Aspects of operator algebras and applications}, volume 534 of
  {\em Contemp. Math.}, pages 1--71. Amer. Math. Soc., Providence, RI, 2011.

\bibitem{BEMSW:Z-2}
S.~Barlak, D.~Enders, H.~Matui, G.~Szab\'o, and W.~Winter.
\newblock {The Rokhlin property vs.\ Rokhlin dimension $1$ on $\mathcal{O}_2$}.
\newblock \emph{J. Noncomm. Geom.}, to appear; arXiv:1312.6289v2, 2014.

\bibitem{BH:JFA}
B.~E. Blackadar and D.~Handelman.
\newblock Dimension functions and traces on {$\mathrm{C}^*$}-algebras.
\newblock {\em J. Funct. Anal.}, 45:297--340, 2982.

\bibitem{BK:MA}
B.~E. Blackadar and E.~Kirchberg.
\newblock Generalized inductive limits of finite-dimensional {$\mathrm{C}^*$}-algebras.
\newblock {\em Math. Ann.}, 207:343--380, 1997.

\bibitem{CE:Ann}
M.-D. Choi and E.~G. Effros.
\newblock The completely positive lifting problem for {$\mathrm{C}^*$}-algebras.
\newblock {\em Ann. of Math. (2)}, 104:585--609, 1976.

\bibitem{CE:IMRN}
A.~Ciuperca and G.~A. Elliott.
\newblock A remark on invariants for {$\mathrm{C}^*$}-algebras of stable rank one.
\newblock {\em IMRN}, article id rnm158, 2008.

\bibitem{C:Ann}
A.~Connes.
\newblock Classification of injective factors. {C}ases {$II_{1},$} {$II_{\infty
  },$} {$III_{\lambda },$} {$\lambda \not=1$}.
\newblock {\em Ann. of Math. (2)}, 104(1):73--115, 1976.

\bibitem{CEI:Crelle}
K.~T. Coward, G.~A. Elliott, and C.~Ivanescu.
\newblock The {C}untz semigroup as an invariant for {$\mathrm{C}^*$}-algebras.
\newblock {\em J. Reine Angew. Math.}, 623:161--193, 2008.

\bibitem{ER:PJM}
E.~G. Effros and J.~Rosenberg.
\newblock {$\mathrm{C}^{\ast} $}-algebras with approximately inner flip.
\newblock {\em Pacific J. Math.}, 77(2):417--443, 1978.

\bibitem{E:JA}
G.~A. Elliott.
\newblock On the classification of inductive limits of sequences of semisimple
  finite-dimensional algebras.
\newblock {\em J. Algebra}, 38:29--44, 1976.

\bibitem{EN:meandim}
G.~A. Elliott and Z. Niu.
\newblock The $\mathrm{C}^{*}$-algebra of a minimal homeomorphism of zero mean dimension.
\newblock arXiv:1406.2382, 2014.

\bibitem{H:JFA}
U.~Haagerup.
\newblock A new proof of the equivalence of injectivity and hyperfiniteness for
  factors on a separable {H}ilbert space.
\newblock {\em J. Funct. Anal.}, 62(2):160--201, 1985.

\bibitem{H:QTrace}
U.~Haagerup.
\newblock Quasitraces on exact {$\mathrm{C}^*$}-algebras are traces.
\newblock {\em C. R. Math. Acad. Sci. Soc. R. Can.}, 36:67-92, 2014.

\bibitem{H:MJM}
D.~Handelman.
\newblock Homomorphisms of {$\mathrm{C}^*$} algebras to finite {$AW^*$} algebras.
\newblock {\em Michigan Math. J.}, 29:229--240, 1981.

\bibitem{HKW:Adv}
I.~Hirshberg, E.~Kirchberg, and S.~White.
\newblock Decomposable approximations of nuclear {$\mathrm{C}^*$}-algebras.
\newblock {\em Adv. Math.}, 230(3):1029--1039, 2012.

\bibitem{JS:AJM}
X.~Jiang and H.~Su.
\newblock On a simple unital projectionless {$\mathrm{C}^*$}-algebra.
\newblock {\em Amer. J. Math.}, 121:359--413, 1999.

\bibitem{Kir:ICM}
E.~Kirchberg.
\newblock {Exact $\mathrm{C}^*$-algebras, tensor products, and the classification of
  purely infinite $\mathrm{C}^{*}$-algebras}.
\newblock In {\em Proceedings of the International Congress of Mathematicians,
  Z\"urich, 1994}, volume 1,2, pages 943--954, Basel, 1995. Birkh\"auser.

\bibitem{K:Abel}
E.~Kirchberg.
\newblock Central sequences in {$\mathrm{C}^*$}-algebras and strongly purely infinite
  algebras.
\newblock In {\em Operator {A}lgebras: {T}he {A}bel {S}ymposium 2004}, volume~1
  of {\em Abel Symp.}, pages 175--231. Springer, Berlin, 2006.

\bibitem{KirRor:pi3}
E.~Kirchberg and M.~R{\o}rdam.
\newblock {Purely infinite $\mathrm{C}^*$-algebras: Ideal preserving zero homotopies}.
\newblock {\em Geom. Funct. Anal.}, 15(2):377--415, 2005.

\bibitem{KR:Crelle}
E.~Kirchberg and M.~R{\o}rdam.
\newblock Central sequence {$\mathrm{C}^*$}-algebras and tensorial absoption of the
  {J}iang-{S}u algebra.
\newblock \emph{J. Reine Angew. Math.}, 695:175--214, 2014.

\bibitem{KW:IJM}
E.~Kirchberg and W.~Winter.
\newblock Covering dimension and quasidiagonality.
\newblock {\em Internat. J. Math.}, 15(1):63--85, 2004.

\bibitem{Lin:crossed-product-AF-embedding}
H.~Lin.
\newblock {AF}-embeddings of the crossed products of {AH}-algebras by finitely
  generated abelian groups.
\newblock {\em Int. Math. Res. Pap. IMRP}, 3:Art. ID rpn007, 67, 2008.

\bibitem{L:Book}
T.~Loring.
\newblock {\em {Lifting Solutions to Perturbing Problems in $\mathrm{C}^*$-algebras}},
  volume~8 of {\em Fields Institute monographs}.
\newblock Amer. Math. Soc., Providence, Rhode Island, 1997.

\bibitem{M:PLMS}
D.~McDuff
\newblock Central sequences and the hyperfinite factor.
\newblock {\em Proc. London Math. Soc. (3)}, 21:443--461, 1970.

\bibitem{MS:Acta}
H.~Matui and Y.~Sato.
\newblock Strict comparison and {$\mathcal {Z}$}-absorption of nuclear
  {$\mathrm{C}^*$}-algebras.
\newblock {\em Acta Math.}, 209(1):179--196, 2012.

\bibitem{MS:Preprint}
H.~Matui and Y.~Sato.
\newblock Decomposition rank of {UHF}-absorbing {$\mathrm{C}^*$}-algebras.
\newblock \emph{Duke Math. J.}, 163(14):2687--2708, 2014.

\bibitem{R:Adv}
L.~Robert.
\newblock Classification of inductive limits of {$1$}-dimensional {NCCW}
  complexes.
\newblock {\em Adv. Math.}, 231:2802--2836, 2012.

\bibitem{R:JFA}
M.~R{\o}rdam.
\newblock On the structure of simple {$\mathrm{C}^*$}-algebras tensored with a
  {UHF}-algebra.
\newblock {\em J. Funct. Anal.}, 100:1--17, 1991.

\bibitem{R:Book}
M.~R{\o}rdam.
\newblock Classification of nuclear, simple {$\mathrm{C}^*$}-algebras.
\newblock In {\em Classification of nuclear {$\mathrm{C}^*$}-algebras. {E}ntropy in
  operator algebras}, volume 126 of {\em Encyclopaedia Math. Sci.}, pages
  1--145. Springer, Berlin, 2002.

\bibitem{R:Acta}
M.~R{\o}rdam.
\newblock A simple {$\mathrm{C}^*$}-algebra with a finite and an infinite projection.
\newblock {\em Acta Math.}, 191:109--142, 2003.

\bibitem{R:IJM}
M.~R{\o}rdam.
\newblock The stable and the real rank of {$\mathcal Z$}-absorbing
  {$\mathrm{C}^*$}-algebras.
\newblock {\em Internat. J. Math.}, 15(10):1065--1084, 2004.

\bibitem{RW:Crelle}
M.~R{\o}rdam and W.~Winter.
\newblock The {J}iang-{S}u algeba revisited.
\newblock {\em J. Reine Angew. Math.}, 642:129--155, 2010.

\bibitem{S:Preprint}
Y.~Sato.
\newblock Discrete amenable group actions on von {Neumann} algebras and
  invariant nuclear {$\mathrm{C}^*$}-algebras.
\newblock arXiv:1104.4339, 2011.

\bibitem{S:Preprint2}
Y.~Sato.
\newblock Trace spaces of simple nuclear {$\mathrm{C}^*$}-algebras with
  finite-dimensional extreme boundary.
\newblock arXiv:1209.3000, 2012.

\bibitem{SS:Book}
A.~Sinclair, and R.~Smith.
\newblock Finite von Neumann algebras and masas.
\newblock \emph{London Mathematical Society Lecture Notes}, Volume 351, 2008.

\bibitem{Sz:Preprint}
G.~Szab\'o.
\newblock The Rokhlin dimension of topological {$\mathbb Z^m$}-actions.
\newblock \emph{Proc. London Math. Soc. (3)}, to appear; arXiv:1308.5418, 2013.

\bibitem{T:1}
M.~Takesaki.
\newblock Theory of operator algebras. I. 
\newblock Springer-Verlag, New York-Heidelberg, 1979.

\bibitem{T:MA}
A.~Tikuisis.
\newblock Nuclear dimension, {$\mathcal Z$}-stability, and algebraic simplicity
  for stably projectionless {$\mathrm{C}^*$}-algebras.
\newblock {\em Math. Ann.}, 358:729--778, 2014.

\bibitem{TW:APDE}
A.~Tikuisis and W.~Winter.
\newblock Decomposition rank of {$\mathcal Z$}-stable {$\mathrm{C}^*$}-algebras.
\newblock \emph{Anal. PDE.}, 7:673--700, 2014.

\bibitem{T:Ann}
A.~Toms.
\newblock On the classification problem for nuclear {$\mathrm{C}^*$}-algebras.
\newblock {\em Ann. of Math. (2)}, 167:1029--1044, 2008.

\bibitem{TWW:Preprint}
A.~Toms, S.~White, and W.~Winter.
\newblock {$\mathcal Z$}-stability and finite dimensional tracial boundaries.
\newblock {\em IMRN}, to appear, arXiv:1209.3292, 2012.


\bibitem{TW:TAMS}
A.~Toms and W.~Winter.
\newblock Strongly self-absorbing {$\mathrm{C}^*$}-algebras.
\newblock {\em Trans. Amer. Math. Soc.}, 359(8):3999--4029, 2007.


\bibitem{TW:CJM}
A.~Toms and W.~Winter.
\newblock $\mathcal{Z}$-stable ASH algebras.
\newblock {\em Canad. J. Math.}, 60(3):703--720, 2008.

\bibitem{TomsWinter:PNAS}
A.~Toms and W.~Winter.
\newblock {Minimal dynamics and the classification of {$\mathrm{C}^*$}-algebras}.
\newblock {\em Proc. Natl. Acad. Sci. USA}, 106(40):16942--16943, 2009.

\bibitem{V:Duke}
D.~Voiculescu.
\newblock A note on quasi-diagonal {$\mathrm{C}^*$}-algebras and homotopy.
\newblock {\em Duke Math. J.}, 62:267--271, 1991.

\bibitem{Win:localizingEC}
W.~Winter.
\newblock {Localizing the Elliott conjecture at strongly self-absorbing
  $\mathrm{C}^{*}$-algebras. With an appendix by H.\ Lin}.
\newblock \emph{J. Reine Angew. Math.}, 692:193--231, 2014.

\bibitem{W:Invent1}
W.~Winter.
\newblock Decomposition rank and {$\mathcal Z$}-stability.
\newblock {\em Invent. Math.}, 179(2):229--301, 2010.

\bibitem{W:JNCG}
W.~Winter.
\newblock Strongly self-absorbing {$\mathrm{C}^*$}-algebras are {$\mathcal Z$}-stable.
\newblock {\em J. Noncomm. Geom.}, 5(2):253--264, 2011.

\bibitem{W:Invent2}
W.~Winter.
\newblock Nuclear dimension and {$\mathcal {Z}$}-stability of pure
  {$\mathrm{C}^*$}-algebras.
\newblock {\em Invent. Math.}, 187(2):259--342, 2012.

\bibitem{Win:class-products}
W.~Winter.
\newblock {Classifying crossed product $\mathrm{C}^*$-algebras}.
\newblock arXiv:1308.5084, 2013.

\bibitem{WZ:MJM}
W.~Winter and J.~Zacharias.
\newblock Completely positive maps of order zero.
\newblock {\em M\"unster J. Math.}, 2:311--324, 2009.

\bibitem{WZ:Adv}
W.~Winter and J.~Zacharias.
\newblock The nuclear dimension of {$\mathrm{C}^\ast$}-algebras.
\newblock {\em Adv. Math.}, 224(2):461--498, 2010.

\bibitem{Wol:disjointness}
M.~Wolff.
\newblock Disjointness preserving operators on {$\mathrm{C}^{*}$}-algebras.
\newblock {\em Arch. Math.}, 62:248--253, 1994.



\end{thebibliography}
\end{document}